\documentclass[a4paper, 10pt, notitlepage]{article}

\usepackage{amsthm, amsmath, amssymb, latexsym}
\usepackage{mathrsfs}

\theoremstyle{plain}
\newtheorem{theorem}{Theorem}[section]
\newtheorem{lemma}{Lemma}[section]

\theoremstyle{definition}
\newtheorem{definition}{Definition}[section]

\theoremstyle{remark}
\newtheorem{remark}{Remark}[section]

\DeclareMathOperator*{\argmin}{arg\,min}
\DeclareMathOperator{\dist}{dist}
\DeclareMathOperator{\dom}{dom}
\DeclareMathOperator{\lineal}{lin}
\DeclareMathOperator{\rank}{rank}
\DeclareMathOperator{\trace}{Tr}

\author{M.V. Dolgopolik}
\title{A Unified Approach to the Global Exactness of Penalty and Augmented Lagrangian Functions II: Extended Exactness.}

\begin{document}

\maketitle

\begin{abstract}
In the second part of our study we introduce the concept of global extended exactness of penalty and augmented
Lagrangian functions, and derive the localization principle in the extended form. The main idea behind the extended
exactness consists in an extension of the original constrained optimization problem by adding some extra variables,
and then construction of a penalty/augmented Lagrangian function for the extended problem. This approach allows one
to design extended penalty/augmented Lagrangian functions having some useful properties (such as smoothness),
which their counterparts for the original problem might not possess. In turn, the global exactness of such extended
merit functions can be easily proved with the use of the localization principle presented in this paper, which reduces
the study of global exactness to a local analysis of a merit function based on sufficient optimality conditions and
constraint qualifications. We utilize the localization principle in order to obtain simple necessary and sufficient
conditions for the global exactness of the extended penalty function introduced by Huyer and Neumaier, and in order to
construct a globally exact continuously differentiable augmented Lagrangian function for nonlinear semidefinite
programming problems.
\end{abstract}

\section{Introduction}

In this two-part study we present a new general approach to the analysis of the global exactness of various
penalty and augmented Lagrangian functions for constrained optimization problems in finite dimensional spaces. This
approach allows one to obtain easily verifiable \textit{necessary and sufficient} conditions for the global exactness of
most of the existing penalty/augmented Lagrangian functions in a simple and straightforward manner with the use of the
so-called \textit{localization principle}. This principle, in essence, reduces the study of the global exactness of a
given merit function to a local analysis of behaviour of this function near globally optimal solutions of the
original problem. In turn, the local analysis can be usually performed with the use of some standard tools of
constrained optimization, such as sufficient optimality conditions and constraint qualifications. Thus, the localization
principle provides one with a simple technique for verifying whether a given penalty/augmented Lagrangian function is
globally exact.

A motivation behind the study of the exactness of penalty/augmented Lagrangian functions, a review of the relevant
literature, as well as some general discussions of the framework for the study of global exactness that is adopted in
our research, are presented in the first paper of the series (see the preprint~\cite{Dolgopolik_UnifiedApproach_I}).

Let us note that there exist several different approaches to the definition of global exactness of
penalty/augmented Lagrangian functions. Each part of this two-part study is devoted to the analysis of one of these
approaches. In this paper, we introduce and investigate the concept of global extended exactness, which naturally
appeared within the theory of \textit{exact} augmented Lagrangian functions and \textit{extended} (or
\textit{parametric}) penalty functions.

The first exact augmented Lagrangian function was introduced by Di Pillo and Grippo \cite{DiPilloGrippo1979} for
equality constrained optimization problems in 1979. This augmented Lagrangians was extended to the case of inequality
constrained optimization problems and thoroughly investigated in
\cite{DiPilloGrippo1980,DiPilloGrippo1982,Lucidi1988,DiPilloLucidiPalagi1993,DiPilloLucidi1996,DiPilloLucidi2001,
DiPilloEtAl2002,DiPilloLiuzzi2003,DuZhangGao2006,DuLiangZhang2006,LuoWuLiu2013,DiPilloLiuzzi2011}. Recently, Di Pillo
and Grippo's exact augmented Lagrangian function was extended to the case of nonlinear semidefinite programming problems
\cite{FukudaLourenco}. A general theory of globally exact augmented Lagrangian functions for cone constrained
optimization problems was developed by the author in \cite{Dolgopolik_GSP}. It should be noted that the main feature of
\textit{exact} augmented Lagrangian functions is the fact that one has to minimize these function in primal and dual
variables \textit{simultaneously} in order to compute KKT-points corresponding to globally optimal solution of the
original constrained problem.

The first exact penalty function depending on some additional parameters, apart from the penalty parameter, (i.e.
\textit{extended} or \textit{parametric} penalty function) was introduced by Huyer and Neumaier \cite{HuyerNeumaier} in
2003. This penalty function was generalized and, later on, applied to various optimization problems in 
\cite{Bingzhuang,WangMaZhou,LiYu,MaLiYiu,LinWuYu,JianLin,LinLoxton,MaZhang2015,ZhengZhang2015,Dolgopolik_OptLet}. 
In~\cite{Dolgopolik_OptLet2}, it was shown that Huyer and Neumaier's extended penalty function is exact if and
only if the standard nonsmooth penalty function is exact, and some relations between the least exact penalty
parameters of these functions were obtained. Finally, the general theory of globally exact extended penalty functions
was developed in \cite{DolgopolikMV_UT_2}. As in the case of exact augmented Lagrangians, the main feature of Huyer and
Neumaier's penalty function is the fact that one has to minimize this function in primal variables and an additional
artificial variable simultaneously in order to recover globally optimal solution of the original problem. 

Thus, both Huyer and Neumaier's penalty function and Di Pillo and Grippo's augmented Lagrangian depend on some extra
variables, and the global exactness of these functions is studied in the extended space including the extra variables.
The introduction of these extra variables allows one to guarantee both smoothness and exactness of the corresponding
penalty/augmented Lagrangian functions. In contrast, standard exact penalty functions are always nonsmooth (see,
e.g., \cite{Dolgopolik_UT}, Remark~3). A straightforward generalization of the main idea behind the aforementioned
penalty and augmented Lagrangian functions to the abstract framework leads to the concept of global extended exactness,
which is the main object of our study.

In this paper, we introduce the concept of global extended exactness and prove the localization principle in the
extended form. With the use of this principle we recover existing simple necessary and sufficient conditions for the
global exactness of Huyer and Neumaier's extended penalty function. We also study the global exactness of a
continuously differentiable augmented Lagrangian function for nonlinear semidefinite programming problems that was
recently introduced by the author \cite{Dolgopolik_GSP}. It should be noted that the theorem on the global exactness of
this augmented Lagrangian function was formulated in \cite{Dolgopolik_GSP} without proof. In this paper we present a
detailed proof of this result.

The paper is organized as follows. In Section~\ref{Section_ExtendedExactness} we introduce the definition of global
extended exactness and derive the localization principle in the extended form. In
Section~\ref{Section_ExtendedExact_Appl} we apply this principle to Huyer and Neumaier's penalty function and to a
continuously differentiable augmented Lagrangian function for nonlinear semidefinite programming problems.

\section{Extended Exactness}
\label{Section_ExtendedExactness}

Let $X$ be a finite dimensional normed space, and $M, A \subset X$ be nonempty sets. Throughout this article, we study
the following optimization problem
$$
  \min f(x) \quad \text{subject to} \quad x \in M, \quad x \in A,
  \eqno{(\mathcal{P})}
$$
where $f\colon X \to \mathbb{R} \cup \{ + \infty \}$ is a given function. Denote by $\Omega = M \cap A$ the set of
feasible points of this problem. From this point onwards, we suppose that there exists $x \in \Omega$ such that
$f(x) < + \infty$, and that there exists a globally optimal solution of $(\mathcal{P})$.

Our aim is to ``get rid'' of the constraint $x \in M$ in the problem $(\mathcal{P})$ with the use of a merit function.
Namely, we want to develop a general theory of merit functions $F(\cdot)$ such that globally optimal solutions of the
problem $(\mathcal{P})$ can be easily recovered from points of global minimum of these functions. 

Let $\Lambda$ be a nonempty set of parameters that are denoted by $\lambda$, and let $c > 0$ be 
\textit{the penalty parameter}. Hereinafter, we suppose that a function 
$F \colon X \times \Lambda \times (0, + \infty) \to \mathbb{R} \cup \{ + \infty \}$, $F = F(x, \lambda, c)$, is given.
A connection between this function and the problem $(\mathcal{P})$ is specified below. 

The function $F$, for instance, can be a penalty function with $\Lambda$ being the empty set or an augmented Lagrangian
function with $\lambda$ being a Lagrange multiplier. However, in order not to restrict ourselves to any specific case,
we call $F(x, \lambda, c)$ \textit{a separating function} for the problem $(\mathcal{P})$. The motivation behind this
term comes from the image space analysis \cite{Giannessi_book} in which penalty and augmented Lagrangian functions are
viewed as nonlinear functions \textit{separating} certain nonconvex sets.

In the first part of our study we analysed the concept of global parametric exactness. Recall that the separating
function $F(x, \lambda, c)$ is called \textit{globally parametrically exact} iff there exists $\lambda^* \in \Lambda$
such that the problem
$$
  \min_x F(x, \lambda^*, c) \quad \text{subject to} \quad x \in A
$$
has the same globally optimal solutions as the problem $(\mathcal{P})$ for any sufficiently large $c > 0$. Any such
$\lambda^* \in \Lambda$ is called \textit{an exact tuning parameter}.

Thus, if a globally parametrically exact separating function $F(x, \lambda, c)$ is constructed, then one can minimize
the function $F(\cdot, \lambda^*, c)$ over the set $A$ in order to find globally optimal solutions of the original
problem, i.e. the function $F(x, \lambda, c)$ allows one to incorporate the constraint $x \in M$ into the objective
function without any loss of information about globally optimal solutions. However, in order to utilize such
function $F(x, \lambda, c)$ one must know an exact tuning parameter $\lambda^* \in \Lambda$ in advance, and the problem
of finding an exact tuning parameter can be even more complicated than the original optimization problem itself
(unless, of course, $F(x, \lambda, c)$ is a penalty function, i.e. unless it does not depend on $\lambda$). For
example, if $F(x, \lambda, c)$ is an augmented Lagrangian function, then $\lambda^*$ is usually a vector of Lagrange
multipliers corresponding to a globally optimal solution of the problem $(\mathcal{P})$. Clearly, in most particular
cases the problem of finding these Lagrange multipliers is at least as difficult as the problem $(\mathcal{P})$ itself.

Furthermore, if $F(x, \lambda, c)$ is globally parametrically exact, then every globally optimal solution of the
problem $(\mathcal{P})$ must be, in particular, a local minimizer of the function $F(\cdot, \lambda^*, c)$ on the set
$A$. In the case when $F(x, \lambda, c)$ is an augmented Lagrangian function, this condition typically results in the
assumption that for any globally optimal solution $x^*$ of the problem $(\mathcal{P})$ the pair $(x^*, \lambda^*)$ is a
KKT-point of this problem. Therefore, in particular, if there exist two globally optimal solutions of the problem
$(\mathcal{P})$ with disjoint sets of Lagrange multipliers, then an augmented Lagrangian function cannot be
globally parametrically exact (cf.~\cite{ShapiroSun}).

In order to avoid complications concerning exact tuning parameters, one can consider a different concept of exactness of
the separating function $F(x, \lambda, c)$. Namely, one can consider the \textit{extended} problem
\begin{equation} \label{ExtendedProblem}
  \min_{x, \lambda} F(x, \lambda, c) \quad \text{subject to} \quad (x, \lambda) \in A \times \Lambda,
\end{equation}
and design a separating function $F(x, \lambda, c)$ such that globally optimal solutions of the original problem
$(\mathcal{P})$ can be recovered from globally optimal solution of the extended problem. This simple idea leads us
to the definition of extended exactness.

The separating function $F(x, \lambda, c)$ is called \textit{globally extendedly exact} iff for any sufficiently 
large $c > 0$ the following two conditions are valid:
\begin{enumerate}
\item{if $(x^*, \lambda^*)$ is a globally optimal solution of the extended problem \eqref{ExtendedProblem}, then
$x^*$ is a globally optimal solution of the problem $(\mathcal{P})$.
}

\item{for any globally optimal solution $x^*$ of the problem $(\mathcal{P})$ there exists $\lambda^* \in \Lambda$ such
that the pair $(x^*, \lambda^*)$ is a globally optimal solution of \eqref{ExtendedProblem}.
}
\end{enumerate}
Thus, if the separating function $F(x, \lambda, c)$ is globally extendedly exact, then one can recover globally optimal
solution of the problem $(\mathcal{P})$ by solving extended problem \eqref{ExtendedProblem} with sufficiently large $c$.

In order to obtain necessary and sufficient conditions for the global extended exactness of the function
$F(x, \lambda, c)$ we need to make two assumptions on the set of parameters $\Lambda$. At first, hereinafter, we suppose
that $\Lambda$ is a closed subset of a finite dimensional normed space. This assumption is needed in order to
ensure that every bounded sequence of parameters $\{ \lambda_n \}$ has a convergent subsequence whose limit point
belongs to $\Lambda$.

The second assumption that we make concerns the nature of globally optimal solutions of the extended problem
\eqref{ExtendedProblem}. Namely, we suppose that one chooses which parameters $\lambda^*$ must correspond to globally
optimal solutions $(x^*, \lambda^*)$ of the extended problem \eqref{ExtendedProblem} in the case when the separating
function $F(x, \lambda, c)$ is globally extendedly exact. We suppose that the choice of parameter $\lambda^*$ is
formulated in the form of the equality constraint $\eta(x^*, \lambda^*) = 0$ with a prespecified 
function $\eta(x, \lambda)$.

Let us give a precise formulation of this assumption. Suppose that a function
$\eta \colon X \times \Lambda \to \mathbb{R}$ is given. We also suppose that for any globally optimal solution $x^*$ of
the problem $(\mathcal{P})$ there exists $\lambda^* \in \Lambda$ such that $\eta(x^*, \lambda^*) = 0$.

\begin{definition}
The separating function $F(x, \lambda, c)$ is called \textit{globally extendedly exact} (with respect to the function
$\eta$) iff there exists $c_0 > 0$ such that for any $c \ge c_0$ the following two conditions are valid:
\begin{enumerate}
\item{if $(x^*, \lambda^*)$ is a globally optimal solution of the extended problem \eqref{ExtendedProblem}, then 
$\eta(x^*, \lambda^*) = 0$, and $x^*$ is a globally optimal solution of the problem $(\mathcal{P})$;
}

\item{for any globally optimal solution $x^*$ of the problem $(\mathcal{P})$ and for any $\lambda^* \in \Lambda$ such
that $\eta(x^*, \lambda^*) = 0$ the pair $(x^*, \lambda^*)$ is a globally optimal solution of the problem
\eqref{ExtendedProblem},
}
\end{enumerate}
The greatest lower bounded of all such $c_0$ is denoted by $c^*_{ext}$, and is called \textit{the least exact penalty
parameter} of the separating function $F(x, \lambda, c)$.
\end{definition}

Let us note that the additional assumption $\eta(x^*, \lambda^*) = 0$ naturally appears in all particular examples 
of globally extendedly exact separating function (see examples below). In particular, if $F(x, \lambda, c)$ is an
augmented Lagrangian function with $\lambda$ being a Lagrange multiplier, then it is natural to require that global
minimizers $(x^*, \lambda^*)$ of the extended problem are exactly KKT-points corresponding to globally optimal solutions
$x^*$ of the problem $(\mathcal{P})$. The assumption that $(x^*, \lambda^*)$ is a KKT-point can be easily expressed in
the form of the equality $\eta(x^*, \lambda^*) = 0$ with a suitable function $\eta$.

Our aim is to obtain simple necessary and sufficient conditions for the global extended exactness of $F(x, \lambda, c)$.
As in the case of parametric exactness, these conditions are formulated in the form of the so-called
\textit{localization principle}. This principle allows one to reduce the study of the global exactness of the separating
function $F(x, \lambda, c)$ to a local analysis of behaviour of this function near globally optimal solutions of the
problem $(\mathcal{P})$. Below, we follow the same line of reasoning as during the derivation of the localization
principle in the parametric form in the first part of our study (see~\cite{Dolgopolik_UnifiedApproach_I}).

As it was mentioned above, the localization principle allows one to study local behaviour of the separating function
$F(x, \lambda, c)$ near globally optimal solution of the problem $(\mathcal{P})$ in order to prove the global exactness
of this functions. The following definition describes desired local behaviour of the function $F(x, \lambda, c)$.

\begin{definition}
Let $x^*$ be a locally optimal solution of the problem $(\mathcal{P})$. The separating function $F(x, \lambda, c)$ is
called \textit{locally extendedly exact} at the point $x^*$ iff for any $\lambda^* \in \Lambda$ such 
that $\eta(x^*, \lambda^*) = 0$ there exist $c_0 > 0$ and a neighbourhood $U$ of the point $(x^*, \lambda^*)$ such that 
$$
  F(x, \lambda, c) \ge F(x^*, \lambda^*, c) \quad 
  \forall (x, \lambda) \in U \cap (A \times \Lambda) \quad \forall c \ge c_0.
$$
The greatest lower bound of all such $c_0$ is denoted by $c^*_{ext}(x^*, \lambda^*)$, and is called \textit{the least
exact penalty parameter} of $F(x, \lambda, c)$ at $(x^*, \lambda^*)$.
\end{definition}

Thus, if the separating function $F(x, \lambda, c)$ is locally exact at a globally optimal solution $x^*$, then for any
$\lambda^* \in \Lambda$ such that $\eta(x^*, \lambda^*) = 0$ and for all $c_0 > c^*_{ext}(x^*, \lambda^*)$ 
the pair $(x^*, \lambda^*)$ is a point of local (uniformly with respect to $c \in [c_0, + \infty)$) minimum of the
function $F(\cdot, \cdot, c)$ on the set $A \times \Lambda$. 

Recall that $c > 0$ is called \textit{the penalty parameter}; however, a connection between the parameter $c > 0$ and
penalization is unclear from the definition of the separating function $F(x, \lambda, c)$. The following
definition specifies this connection.

\begin{definition}
The function function $F(x, \lambda, c)$ is called a \textit{penalty-type} separating function iff there exists 
$c_0 > 0$ such that if
\begin{enumerate}
\item{$\{ c_n \} \subset [c_0, + \infty)$ is an increasing unbounded sequence,
}

\item{$(x_n, \lambda_n) \in \argmin_{(x, \lambda) \in A \times \Lambda} F(x, \lambda, c_n)$ for any $n \in \mathbb{N}$,
}

\item{$(x^*, \lambda^*)$ is a cluster point of the sequence $\{ (x_n, \lambda_n) \}$,
}
\end{enumerate}
then $x^*$ is a globally optimal solution of the problem $(\mathcal{P})$ and $\eta(x^*, \lambda^*) = 0$.
\end{definition}

Thus, roughtly speaking, the separating function $F(x, \lambda, c)$ is of penalty-type iff global minimizers of this
function on the set $A \times \Lambda$ converge to pairs $(x^*, \lambda^*)$ with $x^*$ being a globally
optimal solution of $(\mathcal{P})$ and $\eta(x^*, \lambda^*) = 0$ as the penalty parameter $c > 0$ increases
unboundedly. It should be pointed out that the choice of the term ``penalty-type'' is due to the fact that the behaviour
described in the definition above is characteristic of penalty functions.

Note that if there exists $c_0 > 0$ such that for any $c \ge c_0$ the function $F(\cdot, \cdot, c)$ does not attain a
global minimum on the set $A \times \Lambda$, then, formally, $F(x, \lambda, c)$ is of penalty-type. Similarly, if
sequences of global minimizers $(x_n, \lambda_n) \in \argmin_{(x, \lambda) \in A \times \Lambda} F(x, \lambda, c_n)$, 
$n \in \mathbb{N}$, where $c_n \to + \infty$ as $n \to \infty$, do not have cluster points, then 
the function $F(x, \lambda, c)$ is also of penalty-type. In order to exclude these pathological cases from our
consideration, it is natural to introduce the following definition of a \textit{non-degenerate} separating function.
Recall that $\Lambda$ is a subset of a finite dimensional normed space.

\begin{definition}
The separating function $F(x, \lambda, c)$ is said to be \textit{non-degenerate} iff there exist $c_0 > 0$ and $R > 0$
such that for any $c \ge c_0$ there 
exists $(x(c), \lambda(c)) \in \argmin_{(x, \lambda) \in A \times \Lambda} F(x, \lambda, c)$ with 
$\| x(c) \| + \| \lambda(c)) \| \le R$.
\end{definition}

Roughly speaking, the separating function $F(x, \lambda, c)$ is non-degenerate iff it attains a global minimum 
in $(x, \lambda)$ on the set $A \times \Lambda$ for any sufficiently large $c > 0$, and points of global minimum of 
$F(x, \lambda, c)$ on $A \times \Lambda$ do not escape to infinity as the penalty parameter $c > 0$ increases
unboundedly. Note that the non-degeneracy is a natural assumption ensuring that the definition of a penalty-type
separating function is meaningful.

Now we can formulate and prove the main result of this paper that under some natural assumptions connects the global
extended exactness of the separating function $F(x, \lambda, c)$ with its local extended exactness near globally optimal
solutions of  the problem $(\mathcal{P})$. Denote by $\Omega^*$ the set of globally optimal optimal solutions of the
problem $(\mathcal{P})$.

\begin{theorem}[Localization Principle in the Extended Form I] \label{Thrm_LocPrincipleExtended}
Let $\Lambda$ be a closed subset of a finite dimensional normed space, and let the validity of the conditions
\begin{equation} \label{InclImpliesExtnExactness}
  \eta(x^*, \lambda^*) = 0, \quad (x^*, \lambda^*) \in \argmin_{(x, \lambda) \in A \times \Lambda} F(x, \lambda, c)
\end{equation}
for some $x^* \in \Omega^*$, $\lambda^* \in \Lambda$, and $c > 0$ imply that the separating function $F(x, \lambda, c)$
is globally extendedly exact. Suppose also that the set
\begin{equation} \label{ZeroLevelIsClosed}
  \Big\{ (x, \lambda) \in \Omega^* \times \Lambda \Bigm| \eta(x, \lambda) = 0 \Big\}
\end{equation}
is closed. Then the separating function $F(x, \lambda, c)$ is globally extendedly exact if and only if the following
conditions are valid:
\begin{enumerate}
\item{for any $x^* \in \Omega^*$ there exists $\lambda^* \in \Lambda$ such that $\eta(x^*, \lambda^*) = 0$;}

\item{$F(x, \lambda, c)$ is of penalty-type and non-degenerate;
\label{Nondegneracy_Assumpt}}

\item{$F(x, \lambda, c)$ is locally extendedly exact at every globally optimal solution of the problem $(\mathcal{P})$.
\label{LocalExactness_Assumpt}}
\end{enumerate}
\end{theorem}

\begin{proof}
Let $F(x, \lambda, c)$ be globally extendedly exact. Then, in particular, for any $c > c^*_{ext}$,
and for all $x^* \in \Omega^*$ and $\lambda^* \in \Lambda$ such that $\eta(x^*, \lambda^*) = 0$ 
the pair $(x^*, \lambda^*)$ is a global minimizer of the function $F(x, \lambda, c)$ in $(x, \lambda)$ on 
the set $A \times \Lambda$. Therefore $F(x, \lambda, c)$ is non-degenerate with $R = \| x^* \| + \| \lambda^* \|$, and
locally extendedly exact at every globally optimal solution of the problem $(\mathcal{P})$.

Let, now, $\{ c_n \} \subset (c^*_{ext}, + \infty)$ be an increasing unbounded sequence, and let a sequence
$\{ (x_n, \lambda_n) \}$ be such  
that $(x_n, \lambda_n) \in \argmin_{(x, \lambda) \in A \times \Lambda} F(x, \lambda, c_n)$
for all $n \in \mathbb{N}$. From the fact that $F(x, \lambda, c)$ is globally extendedly exact it follows that 
for any $n \in \mathbb{N}$ one has $x_n \in \Omega^*$ and $\eta(x_n, \lambda_n) = 0$. Hence taking into account the fact
that the set \eqref{ZeroLevelIsClosed} is closed one gets that any cluster point $(x^*, \lambda^*)$ of the sequence 
$\{ (x_n, \lambda_n) \}$, if exists, satisfies the conditions $x^* \in \Omega^*$ and $\eta(x^*, \lambda^*) = 0$, which
implies that $F(x, \lambda, c)$ is a penalty-type separating function. Thus, the ``only if'' part of the theorem is
proved. Let us prove the ``if'' part.

Let $\{ c_n \} \subset (0, + \infty)$ be an increasing unbounded sequence. By condition~\ref{Nondegneracy_Assumpt}
there exist $n_0 \in \mathbb{N}$ and $R > 0$ such that for any $n \ge n_0$ there exists 
$(x_n, \lambda_n) \in \argmin_{(x, \lambda) \in A \times \Lambda} F(x, \lambda, c)$ with 
$\| x_n \| + \| \lambda_n \| \le R$. Since the sequence $\{ (x_n, \lambda_n) \}$, $n \ge n_0$, is bounded, and 
$A \times \Lambda$ is a subset of a finite dimensional normed space, there exists a subsequence $\{ (x_{n_k},
\lambda_{n_k}) \}$ converging to some $(x^*, \lambda^*)$. 

By condition~\ref{Nondegneracy_Assumpt} the separating function $F(x, \lambda, c)$ is of penalty-type. 
Therefore $x^* \in \Omega^*$ and $\eta(x^*, \lambda^*) = 0$. Applying condition~\ref{LocalExactness_Assumpt} one obtains
that there exist $\widehat{c} > 0$ and a neighbourhood $U$ of the pair $(x^*, \lambda^*)$ such that
\begin{equation} \label{LocalExtendExactness}
  F(x, \lambda, c) \ge F(x^*, \lambda^*, c) \quad \forall (x, \lambda) \in U \cap (A \times \Lambda) 
  \quad \forall c \ge \widehat{c}.
\end{equation}
From the fact that $\{ c_n \}$ is an increasing unbounded sequence it follows that $c_n > \widehat{c}$ for any $n$ large
enough. Furthermore, one has $(x_{n_k}, \lambda_{n_k}) \in U$ for any sufficient large $k$ due to the fact that
$(x^*, \lambda^*)$ is a limit point of the subsequence $\{ (x_{n_k}, \lambda_{n_k}) \} \subset A \times \Lambda$.
Consequently, applying \eqref{LocalExtendExactness} one obtains that there exists $k_0 \in \mathbb{N}$ such that 
$$
  F(x_{n_k}, \lambda_{n_k}, c_{n_k}) \ge F(x^*, \lambda^*, c_{n_k}) \quad \forall k \ge k_0,
$$
which implies that $(x^*, \lambda^*)$ is a point of global minimum of $F(x, \lambda, c_{n_k})$ in $(x, \lambda)$ on 
the set $A \times \Lambda$ for any $k \ge k_0$ by virtue of the definition of the sequence $\{ (x_n, \lambda_n) \}$.
Thus, the triplet $(x^*, \lambda^*, c_{n_k})$ satisfies conditions \eqref{InclImpliesExtnExactness} for any $k \ge k_0$,
which implies that $F(x, \lambda, c)$ is globally extendedly exact.
\end{proof}

\begin{remark}
{(i)~The assumptions that the validity of \eqref{InclImpliesExtnExactness} implies the global extended exactness of
$F(x, \lambda, c)$ simply means that instead of verifying that the sets
$\{ (x^*, \lambda^*) \in \Omega^* \times \Lambda \mid \eta(x^*, \lambda^*) = 0 \}$ and 
$\argmin_{(x, \lambda) \in A \times \Lambda} F(x, \lambda, c)$ coincide, it is sufficient to check that these sets only
intersect in order to prove the global extended exactness of the separating function $F(x, \lambda, c)$. Let us note
that this assumption is automatically satisfied for all particular examples of the function
$F(x, \lambda, c)$ (see examples below). Thus, this assumption is not restrictive in all important cases.
}

\noindent{(ii)~Note that the ``only if'' part of the theorem above is valid without the assumption that conditions
\eqref{InclImpliesExtnExactness} implies the global extended exactness $F(x, \lambda, c)$.
}

\noindent{(iii)~It is easily seen the set \eqref{ZeroLevelIsClosed} is closed, in particular, if $\Omega$ is
closed, $f$ is l.s.c. on $\Omega$, and $\eta$ is continuous on $\Omega \times \Lambda$. 
}
\end{remark}

Let us also present a slightly different formulation of the localization principle in the extended form that is more
convenient for applications.

\begin{theorem}[Localization Principle in the Extended Form II] \label{Thrm_LocPrincipleExtended_SubLevel}
Let $\Lambda$ be a closed subset of a finite dimensional normed space, and let the validity of the
conditions \eqref{InclImpliesExtnExactness} for some $x^* \in \Omega^*$, $\lambda^* \in \Lambda$, and $c > 0$ implies
that the separating function $F(x, \lambda, c)$ is globally extendedly exact. Suppose, finally, that the sets $A$ and
$\{ (x, \lambda) \in \Omega^* \times \Lambda \mid \eta(x, \lambda) = 0 \}$ are closed, and 
the function $F(\cdot, \cdot, c)$ is l.s.c. on $A \times \Lambda$ for any $c > 0$. Then the separating 
function $F(x, \lambda, c)$ is globally extendedly exact if and only if the following conditions are valid:
\begin{enumerate}
\item{for any $x^* \in \Omega^*$ there exists $\lambda^* \in \Lambda$ such that $\eta(x^*, \lambda^*) = 0$;}

\item{$F(x, \lambda, c)$ is of penalty-type;}

\item{$F(x, \lambda, c)$ is locally extendedly exact at every globally optimal solution of the problem $(\mathcal{P})$;}

\item{there exist $c_0 > 0$, $x^* \in \Omega^*$, $\lambda^* \in \Lambda$, and a bounded 
set $K$ such that $\eta(x^*, \lambda^*) = 0$ and
$$
  S_c(x^*, \lambda^*) := \Big\{ (x, \lambda) \in A \times \Lambda \Bigm| F(x, \lambda, c) < F(x^*, \lambda^*, c) \Big\}
  \subseteq K 
  \quad \forall c \ge c_0.
$$
\label{SublevelBoundedness_Assumption}}
\end{enumerate}
\end{theorem}

\begin{proof}
Let us prove the ``if'' part of the theorem. The validity of the ``only if'' part of the theorem follows directly from 
Theorem~\ref{Thrm_LocPrincipleExtended}, and the fact that if $F(x, \lambda, c)$ is globally extendedly exact, then the
set $S_c(x^*, \lambda^*)$ is empty for any $c > c^*_{ext}$, $x^* \in \Omega^*$ and $\lambda^* \in \Lambda$ such 
that $\eta(x^*, \lambda^*) = 0$.

Let $x^* \in \Omega^*$ and $\lambda^* \in \Lambda$ be from condition~\ref{SublevelBoundedness_Assumption}. Observe that
if the set $S_c(x^*, \lambda^*)$ is empty for some $c \ge c_0$, then conditions \eqref{InclImpliesExtnExactness} are
satisfied, and one obtains that $F(x, \lambda, c)$ is globally extendedly
exact. Consequently, one can suppose that $S_c(x^*, \lambda^*) \ne \emptyset$ for all $c \ge c_0$. Taking into account
the facts that $F(\cdot, \cdot, c)$ is l.s.c. on $A \times \Lambda$, and the set $A \times \Lambda$ is closed one gets
that for any $c \ge c_0$ the function $F(x, \lambda, c)$ attains a global minimum in $(x, \lambda)$ on
the set $A \times \Lambda$. Furthermore, for all $c \ge c_0$ every point of global minimum of $F(x, \lambda, c)$ in 
$(x, \lambda)$ on $A \times \Lambda$ belongs to a bounded set $K$. Therefore the function $F(x, \lambda, c)$ is
non-degenerate. Then applying Theorem~\ref{Thrm_LocPrincipleExtended} one obtains the desired result.
\end{proof}

\begin{remark}
The fourth condition in the theorem above may seem superficial in comparison with the rather natural non-degeneracy
condition. However, as we show below, this condition becomes more convenient than the non-degeneracy assumption in
all important particular cases.
\end{remark}

\section{Applications of the Localization Principle}
\label{Section_ExtendedExact_Appl}

Below, we present two particular examples of separating functions that are globally extendedly exact, and demonstrate
how one can utilize the localization principle to obtain necessary and sufficient conditions for the global extended
exactness of these separating function in a simple and straightforward manner.

\subsection{Example I: Huyer and Neumaier's Penalty Function}

Let us apply the localization principle in the extended from to a simple modification of the exact
penalty function proposed by Huyer and Neumaier in \cite{HuyerNeumaier}. We call this penalty functions
\textit{singular}. For theoretical results on the exactness of singular penalty functions as well as applications of
these functions to various optimization problems, see
\cite{HuyerNeumaier,Bingzhuang,WangMaZhou,LiYu,MaLiYiu,LinWuYu,JianLin,LinLoxton,MaZhang2015,ZhengZhang2015,
Dolgopolik_OptLet,Dolgopolik_OptLet2,DolgopolikMV_UT_2}.

Let the set $M$ have the form $M = \{ x \in X \mid 0 \in G(x) \}$, where $G \colon X \rightrightarrows Y$ is a given
set-valued mapping with closed values, and $Y$ is a normed space (not necessarily finite dimensional). Let, also,
$\Lambda = \mathbb{R}_+ = [0, + \infty)$. In order to distinguish points of the set $\Lambda$ from Lagrange
multipliers, in this subsection we denote them as $p$.

Fix arbitrary $w \in Y$, and choose nondecreasing functions $\phi \colon [0, + \infty] \to [0, + \infty]$ and
$\omega \colon \mathbb{R}_+ \to [0, +\infty]$ such that $\phi(t) = 0$ iff $t = 0$ and $\omega(t) = 0$ iff $t = 0$.
Following the ideas of \cite{Dolgopolik_OptLet2,DolgopolikMV_UT_2}, define the singular penalty function
$$
  F(x, p, c) = \begin{cases}
    f(x) + \dfrac{c}{p} \phi\Big(\dist^2\big( 0, G(x) - p w \big) \Big) + c \omega(p), & \text{if }	p > 0, \\
    f(x), & \text{if } p = 0, \: x \in \Omega, \\
    + \infty, & \text{if } p = 0, \: x \notin \Omega.
  \end{cases}
$$
Note that
\begin{equation} \label{InfinitePenaltyFunc}
  F(x, 0, c) = \begin{cases}
    f(x), & \text{if $x$ is feasible}, \\
    + \infty, & \text{otherwise}.
  \end{cases}
\end{equation}
Consequently, the problem of minimizing the function $F(x, 0, c)$ over the set $A$ is equivalent to the problem
$(\mathcal{P})$. Furthermore, one can verify that under very mild additional assumptions
$F(x, p, c) \to F(x, 0, c)$ as $p \to + 0$ for any $x \in X$ and $c > 0$. In addition, if the functions
$f$, $\phi$ and $\omega$ are continuously differentiable on their domains, and the multifunction $G$ is actually
single-valued and Fr\'echet differentiable, then the singular penalty function $F(x, p, c)$ is continuously
differentiable at every point $(x, p) \in \dom F(\cdot, \cdot, c)$ such that $p > 0$. Thus, for any $p > 0$ the
function $F(x, p, c)$ can be viewed as a continuously differentiable approximation of the function $F(x, 0, c)$.
Finally, let us note that the vector $w$ is added into the definition of $F(x, p, c)$ in order for this penalty function
to resemble the Hestenes-Powell-Rockafellar augmented Lagrangian function (see~\cite{HuyerNeumaier} for more details,
and \cite{Dolgopolik_OptLet2} for some results on a connection between the choice of $w$ and the value of the least
exact penalty parameter of the function $F(x, p, c)$).

Our aim is to demonstrate that under some natural assumptions the singular penalty function $F(x, p, c)$ is
globally extendedly exact with $\eta(x, p) = p$, i.e. for any sufficiently large $c$ the problem
$$
  \min_{(x, p)} F(x, p, c) \quad (x, p) \in A \times \mathbb{R}_+
$$
is equivalent to the original problem $(\mathcal{P})$ in the sense that any globally optimal solution of the above
problem has the form $(x^*, 0)$ with $x^*$ being a globally optimal solution of the problem $(\mathcal{P})$. We utilize
the localization principle in the extended from in order to obtain this result. Define $\eta(x, p) = p$.

\begin{theorem}[Localization Principle for Singular Penalty Functions] \label{Thrm_LocPrinciple_SingPF}
Let $A$ be closed, $f$ be l.s.c. on $A$, $\phi$ and $\omega$ be l.s.c., and $G$ be outer semicontinuous on $A$. Then the
singular penalty function $F(x, p, c)$ is globally extendedly exact if and only if it is locally extendedly
exact at every globally optimal solution of the problem $(\mathcal{P})$, and one of the following two conditions is
satisfied:
\begin{enumerate}
\item{the function $F(x, p, c)$ is non-degenerate;
}

\item{there exists $c_0 > 0$ such that the set $\{ (x, p) \in A \times \mathbb{R}_+ \mid F(x, p, c_0) < f^* \}$ is
bounded, where $f^* = \inf_{x \in \Omega} f(x)$ is the optimal value of the problem $(\mathcal{P})$.
}
\end{enumerate}
\end{theorem}

\begin{proof}
Taking into account \eqref{InfinitePenaltyFunc}, and the fact that for any $p > 0$ either $F(x, p, c)$ is strictly
increasing in $c$ or $F(x, p, \cdot) \equiv + \infty$ it is easy to check that the validity of the condition 
$(x^*, 0) \in \argmin_{(x, \lambda) \in A \times \Lambda} F(x, \lambda, c)$ implies that $F(x, p, c)$ is globally
extendedly exact.

Applying \cite[Lemma~4.2 and Remark~12]{DolgopolikMV_UT_2} one gets that under the assumptions of the theorem the
singular penalty function is l.s.c. in $(x, p)$ on $A \times \mathbb{R}_+$ for any $c > 0$. Therefore it remains to
check that $F(x, p, c)$ is a penalty-type separating function. Then applying the localization principle in the extended
form (Theorems~\ref{Thrm_LocPrincipleExtended} and \ref{Thrm_LocPrincipleExtended_SubLevel}) one obtains the desired
result.

Let $\{ c_n \} \subset (0, + \infty)$ be an increasing unbounded sequence, and let 
$(x_n, p_n) \in \argmin_{(x, p) \in A \times \mathbb{R}_+} F(x, p, c_n)$ for all $n \in \mathbb{N}$. Suppose also that
$(x^*, p^*)$ is a cluster point of the sequence $\{ (x_n, p_n) \}$. Applying \cite[Theorem~4.13]{DolgopolikMV_UT_2} one
obtains that $p_n \to 0$, and $\dist( 0, G(x_n)) \to 0$ as $n \to \infty$. Hence $p^* = 0$, and taking into account the
outer semicontinuity of the multifunction $G$ one can easily check that $0 \in G(x^*)$, i.e. $x^*$ is a feasible point
of the problem $(\mathcal{P})$.

Since $F(x^*, 0, c) = f^*$ for any $x^* \in \Omega^*$, one has $F(x_n, p_n, c_n) \le f^*$, and $f(x_n) \le f^*$ due to
the fact that $F(x, p, c) \ge f(x)$ for all $x \in X$, $p \in \mathbb{R}_+$ and $c > 0$ by the definition of 
$F(x, p, c)$. Cosequently, passing to the limit as $n \to \infty$, and applying the lower semicontinuity of the function
$f$ one obtains that $x^*$ is a globally optimal solution of the problem $(\mathcal{P})$, which implies that 
$F(x, p, c)$ is a penalty-type separating function.
\end{proof}

\begin{remark}
{(i)~For the sake of completeness, let us note that the singular penalty function $F(x, p, c)$ is locally
extendedly exact at a globally optimal solution $x^*$ of the problem $(\mathcal{P})$, provided $f$ is Lipschitz
continuous near $x^*$, $G$ is metrically subregular at $(x^*, 0)$, and there exist $\phi_0 > 0$, $\omega_0 > 0$ and 
$t_0 > 0$ such that $\phi(t) \ge \phi_0 t$ and $\omega(t) \ge \omega_0 t$ for all $t \in [0, t_0]$
(see~\cite[Theorem~4.1]{DolgopolikMV_UT_2}).
}

\noindent{(ii)~One can verify that under some natural assumptions on the functions $\phi$ and $\omega$ the set 
$\{ (x, p) \in A \times \mathbb{R}_+ \mid F(x, p, c_1) < f^* \}$ is bounded, in particular, if the set
$\{ x \in A \mid f(x) + c_2 d(0, G(x)) < f^* \}$ is bounded for some $c_2 > 0$ \cite[Lemma~4.5]{DolgopolikMV_UT_2}.
}

\noindent{(iii)~It is worth noting that under some mild assumptions on the functions $\phi$ and $\omega$ the singular
penalty function $F(x, p, c)$ is globally extendedly exact iff the standard penalty function 
$F_0(x, c) = f(x) + c \dist(0, G(x))$ for the problem ($\mathcal{P})$ is globally exact (see~\cite{Dolgopolik_OptLet2}).
}
\end{remark}

\subsection{Example II: An Exact Augmented Lagrangian Function for Semidefinite Optimization}

In this section, we apply the localization principle in the extended form to an \textit{exact} augmented Lagrangian
function. The first exact augmented Lagrangian function was introduced by Di Pillo and Grippo in
\cite{DiPilloGrippo1979}, and later on was improved and thoroughly investigated by many researchers 
\cite{DiPilloGrippo1979,DiPilloGrippo1980,DiPilloGrippo1982,Lucidi1988,DiPilloLucidiPalagi1993,DiPilloLucidi1996,
DiPilloLucidi2001,DiPilloEtAl2002,DiPilloLiuzzi2003,DuZhangGao2006,DuLiangZhang2006,LuoWuLiu2013,DiPilloLiuzzi2011,
FukudaLourenco}. The general theory of exact augmented Lagrangian functions for cone constrained optimization problems
was developed by the author in \cite{Dolgopolik_GSP}.

The main goal of this section is to introduce a continuously differentiable exact augmented Lagrangian function for
nonlinear semidefinite programming problem, and to prove its global extended exactness with the use of the localization
principle. This augmented Lagrangian function was first introduced by the author in \cite{Dolgopolik_GSP}; however, the
paper \cite{Dolgopolik_GSP} does not contain a proof of the global exactness of this augmented Lagrangian. Here we
present a detailed and almost self-contained (apart from some technical results from semidefinite optimization) proof of
this result.

Let us note that a different exact augmented Lagrangian function for semidefinite programs was earlier introduced in
\cite{FukudaLourenco}. However, it should be underlined that our augmented Lagrangian function is defined via the
problem data directly, while the augmented Lagrangian function from \cite{FukudaLourenco} depends on a solution of a
certain system of linear equations. Furthermore, in order to correctly define the augmented Lagrangian function from
\cite{FukudaLourenco} one must suppose that \textit{every} feasible point of the nonlinear semidefinite program is
nondegenerate \cite[Def.~4.70]{BonnansShapiro}, which might be a too restrictive assumption for many applications.
In contrast, we assume that only globally optimal solutions of the problem under consideration are nondegerate.

Let $X = A = \mathbb{R}^d$, and suppose that the set $M$ has the form
$$
  M = \Big\{ x \in \mathbb{R}^d \Bigm| G(x) \preceq 0, \: h(x) = 0 \Big\},
$$
where $G \colon X \to \mathbb{S}^l$ and $h \colon X \to \mathbb{R}^s$ are given functions, $\mathbb{S}^l$ is the set of
all $l \times l$ real symmetric matrices, and the relation $G(x) \preceq 0$ means that the matrix $G(x)$ is negative
semidefinite. We suppose that the space $\mathbb{S}^l$ is equipped with the Frobenius norm 
$\| A \|_F = \sqrt{\trace(A^2)}$. Note that this norm corresponds to the inner product 
$\langle A, B \rangle = \trace(A B)$. In this case the problem $(\mathcal{P})$ is a nonlinear semidefinite programming
problem of the form
$$
  \min f(x) \quad \text{subject to} \quad G(x) \preceq 0, \quad h(x) = 0.
$$
Suppose that the functions $f$, $G$ and $h$ are continuously differentiable. For any $\lambda \in \mathbb{S}^l$ 
and $\mu \in \mathbb{R}^s$ denote by
$$
  L(x, \lambda, \mu) = f(x) + \trace( \lambda G(x) ) + \langle \mu, h(x) \rangle
$$
the standard Lagrangian function for the nonlinear semidefinite programming problem. For the sake of shortness we
will sometimes denote $\nu = (\lambda, \mu)$. 

Our aim is to introduce a continuously differentiable augmented Lagrangian function $\mathscr{L}(x, \lambda, \mu, c)$
for the problem $(\mathcal{P})$ that is globally extendedly exact with respect to a function $\eta(x, \lambda, \mu)$
such that $\eta(x^*, \lambda^*, \mu^*) = 0$ for some $x^* \in \Omega^*$ iff $(x^*, \lambda^*, \mu^*)$ is a KKT-point of
the problem $(\mathcal{P})$. In this case one obtains that the augmented Lagrangian 
function $\mathscr{L}(x, \lambda, \mu, c)$ is globally extendedly exact iff its points of global minimum are exactly
KKT-points of the problem $(\mathcal{P})$ corresponding to globally optimal solutions of this problem.

Define
$$
  \eta(x, \lambda, \mu) = \big\| \nabla_x L(x, \lambda, \mu) \big\|^2 + \trace( \lambda^2 G(x)^2).
$$
In order to ensure that $\eta(x^*, \lambda^*, \mu^*) = 0$ iff $(x^*, \lambda^*, \mu^*)$ is a KKT-point of the problem
$(\mathcal{P})$ we need to utilize a proper constraint qualification. 

Let $x^*$ be a locally optimal solution of the problem $(\mathcal{P})$. Recall that the point $x^*$ is called
\textit{nondegenerate} \cite[Def.~4.70]{BonnansShapiro} iff
$$
  \begin{bmatrix}
    D G(x^*) \\
    \nabla h(x^*)
  \end{bmatrix} \mathbb{R}^d +
  \begin{bmatrix}
    \lineal T_{\mathbb{S}^l_{-}} \big( G(x^*) \big) \\
    \{ 0 \}
  \end{bmatrix} =
  \begin{bmatrix}
    \mathbb{S}^l \\
    \mathbb{R}^s
  \end{bmatrix},
$$
where $D G(x^*)$ is the Fr\'{e}chet derivative of $G(\cdot)$ at the point $x^*$, ``lin'' stands for the lineality
subspace of a convex cone, i.e. the largest linear space contained in this cone, and $T_{\mathbb{S}^l_{-}} ( G(x^*) )$
is the contingent cone to the cone of $l \times l$ negative semidefinite matrices $\mathbb{S}^l_{-}$ at the point
$G(x^*)$. Let us note that the nondegenracy condition guarantees that there exists a \textit{unique} Lagrange multiplier
at $x^*$ \cite[Proposition~4.75]{BonnansShapiro}.

The above nondegeneracy condition can be rewritten as a linear independen\-ce-type condition. Namely, 
let $\rank G(x^*) = r$. Then the point $x^*$ is nondegenerate iff the $d$-dimensional vectors
\begin{equation} \label{Nondegeneracy_SemiDef}
  v_{ij} = 
  \left( e_i^T D_{x_1} G(x^*) e_j, \ldots, e_i^T D_{x_d} G(x^*) e_j \right)^T,
  \quad	\nabla h_k(x^*)
\end{equation}
are linearly independent, where $1 \le i \le j \le l - r$, $e_1, \ldots e_{l - r}$ is a basis of the null space of the
matrix $G(x^*)$, $1 \le k \le s$, and $h(x) = (h_1(x), \ldots, h_s(x))$ (see \cite[Proposition~5.71]{BonnansShapiro}).

Note that by \cite[Lemma~4]{Dolgopolik_GSP} the nondegeneracy conditions guarantees that 
$\eta(x^*, \lambda^*, \mu^*) = 0$ iff $(x^*, \lambda^*, \mu^*)$ is a KKT-point of the problem $(\mathcal{P})$, and the
matrix $D^2_{\nu \nu} \eta(x^*, \nu^*)$, where $\nu^* = (\lambda^*, \mu^*)$, is positive definite.

Let us introduce an augmented Lagrangian function for nonlinear semidefinite programming problems. Choose 
$\alpha > 0$ and $\varkappa \ge 1$, and define
$$
  p(x, \lambda) = \frac{a(x)}{1 + \trace(\lambda^2)}, \quad
  q(x, \mu) = \frac{b(x)}{1 + \| \mu \|^2},
$$
where
$$
  a(x) = \alpha - \trace\big( [ G(x) ]_+^2 \big)^{\varkappa}, \quad
  b(x) = \alpha - \| h(x) \|^2,
$$
and $[\cdot]_+$ is the projection of a matrix onto the cone of $l \times l$ positive semidefinite matrices. Denote
$\Omega_{\alpha} = \{ x \in \mathbb{R}^d \mid a(x) > 0, \: b(x) > 0 \}$, and define
\begin{multline} \label{EAL_SemiDef}
  \mathscr{L}(x, \lambda, \mu, c) = f(x) 
  + \frac{1}{2c p(x, \lambda)} \Big( \trace\big( [c G(x) + p(x, \lambda) \lambda]_+^2 \big) - 
  p(x, \lambda)^2 \trace(\lambda^2) \Big) \\
  + \langle \mu, h(x) \rangle + \frac{c}{2 q(x, \mu)} \| h(x) \|^2 + \eta(x, \lambda, \mu),
\end{multline}
if $x \in \Omega_{\alpha}$, and $\mathscr{L}(x, \lambda, \mu, c) = + \infty$, otherwise. It is easy to see that the
function $\mathscr{L}(\cdot, c)$ is lower semicontinuous for all $c > 0$. Furthermore, one can verify that
$\mathscr{L}(x, \lambda, \mu, c)$ is continuously differentiable on its effective domain, provided the functions $f$,
$G$ and $h$ are twice continuously differentiable (cf.~\cite[Sect.~3]{ShapiroSun}). Note also that the augmented
Lagrangian \eqref{EAL_SemiDef} is constructed from a straightforward modification of the Hestenes-Powell-Rockafellar
augmented Lagrangian to the case of semidefinite programming problems
\cite{SunZhangWu2006,SunSunZhang2008,ZhaoSunToh2010,LuoWuChen2012,WuLuoDingChen2013,WuLuoYang2014}

Let us obtain simple necessary and sufficient conditions for the global extended exactness of the augmented Lagrangian
function \eqref{EAL_SemiDef}. Denote $\Lambda = \mathbb{S}^l \times \mathbb{R}^s$.

\begin{theorem}[Localization Principle for Exact Augmented Lagrangian Functions] \label{Thrm_EAL_SemiDef_GlobalExact}
Let the functions $f$, $G$, and $h$ be continuously differentiable. Suppose also that every globally optimal
solution of the problem $(\mathcal{P})$ is nondegenerate. Then the augmented Lagrangian function
$\mathscr{L}(x, \lambda, \mu, c)$ is globally extendedly exact if and only if it is locally extendedly exact at every
globally optimal solution of the problem $(\mathcal{P})$, and one of the two following conditions is satisfied:
\begin{enumerate}
\item{the function $\mathscr{L}(x, \lambda, \mu, c)$ is non-degenerate;}

\item{the set $\{ (x, \lambda, \mu) \in \mathbb{R}^d \times \Lambda \mid \mathscr{L}(x, \lambda, \mu, c_0) < f^* \}$ is
bounded for some $c_0 > 0$;
}
\end{enumerate}
In particular, if the set 
$\Omega(\alpha, \gamma) = \{ x \in \mathbb{R}^d \mid f(x) < f^* + \gamma, \: a(x) > 0, \: b(x) > 0 \}$ is
bounded for some $\gamma > 0$, then the augmented Lagrangian function $\mathscr{L}(x, \lambda, \mu, c)$ is globally
extendedly exact if and only if it is locally extendedly exact at every globally optimal solution of the problem
$(\mathcal{P})$.
\end{theorem}

We divide the proof of the theorem above into three lemmas.

\begin{lemma} \label{Lemma_EAL_IntersectImpliesExact}
Let the functions $f$, $G$, and $h$ be continuously differentiable, and  let every globally optimal solution of the
problem $(\mathcal{P})$ be nondegenerate. Then the validity of the conditions
\begin{equation} \label{EAL_IntersectionImpliesExactness}
  (x^*, \lambda^*, \mu^*) \in \argmin_{(x, \lambda, \mu)} \mathscr{L}(x, \lambda, \mu, c_0), 
  \quad \eta(x^*, \lambda^*, \mu^*) = 0
\end{equation}
for some $x^* \in \Omega^*$, $(\lambda^*, \mu^*) \in \Lambda$ and $c_0 > 0$ implies that the augmented Lagrangian
function $\mathscr{L}(x, \lambda, \mu, c)$ is globally extendedly exact. 
\end{lemma}

\begin{proof}
Introduce the function 
\begin{equation} \label{EAL_InequalPenTerm_Represent}
  \Phi(x, \lambda, c) = \min_{y \in \mathbb{S}^l_{-} - G(x)}
  \Big( - p(x, \lambda) \langle \lambda, y \rangle + \frac{c}{2} \| y \|_F^2 \Big),
\end{equation}
where $\langle \lambda, y \rangle = \trace(\lambda y)$ is the inner product in $\mathbb{S}^l$. Then for any 
$x \in \Omega_{\alpha}$ one has
\begin{equation} \label{EAL_Representation}
  \mathscr{L}(\xi, c) = f(x) + \frac{1}{p(x, \lambda)} \Phi(x, \lambda, c) +
  \langle \mu, h(x) \rangle + \frac{c}{2 q(x, \mu)} \| h(x) \|^2 + \eta(\xi),
\end{equation}
where $\xi = (x, \lambda, \mu)$ (see~\cite{ShapiroSun}, formulae $(2.5)$ and $(2.9)$). Therefore the function
$\mathscr{L}(\xi, c)$ is nondecreasing in $c$.

Suppose that \eqref{EAL_IntersectionImpliesExactness} holds true. Then by our assumption the point $x^*$ is
nondegenerate, and $\eta(x^*, \lambda^*, \mu^*) = 0$. With the use of \cite[Lemma~4]{Dolgopolik_GSP} one obtains
that $(x^*, \lambda^*, \mu^*)$ is a KKT-point of the problem $(\mathcal{P})$. Hence, in particular, 
$\langle \lambda^*, G(x^*) \rangle = 0$, and the matrix $\lambda^*$ is positive semidefinite. Consequently, applying the
standard first order necessary and sufficient conditions for a minimum of a convex function on a convex set one can easy
check that the minimum in
\begin{multline*}
  [c G(x^*) + p(x^*, \lambda^*) \lambda^*]_+^2 = 
  \dist^2 \big( c G(x^*) + p(x^*, \lambda^*) \lambda^*, \mathbb{S}^l_{-} \big) \\
  = \min_{z \in \mathbb{S}^l_{-}} \big\| c G(x^*) + p(x^*, \lambda^*) \lambda^* - z \big\|^2
\end{multline*}
is attained at the point $z = c G(x^*)$. Here we used the equality $\| P_{K^*}(y) \| = \dist(y, K)$ that is valid for
any closed convex cone $K$, where $P_{K^*}(\cdot)$ is the projection operator onto the polar cone $K^*$ of the cone $K$.
Thus, one has $\mathscr{L}(x^*, \lambda^*, \mu^*, c) = f^*$ for any $c > 0$ (see~\eqref{EAL_SemiDef}), which
implies that
\begin{equation} \label{EAL_GlobMinValue}
  \min_{(x, \lambda, \mu)} \mathscr{L}(x, \lambda, \mu, c) = f^* \quad \forall c \ge c_0
\end{equation}
due to the fact that the function $\mathscr{L}(\xi, c)$ is nondecreasing in $c$.

Let, now, $\overline{x} \in \Omega^*$ be arbitrary. Since $\overline{x}$ is nondegenerate, there exists a unique
pair $(\overline{\lambda}, \overline{\mu})$ such that the tripler $(\overline{x}, \overline{\lambda}, \overline{\mu})$
is a KKT-point of the problem $(\mathcal{P})$ \cite[Proposition~4.75]{BonnansShapiro}, and 
$\eta(\overline{x}, \lambda, \mu) = 0$ iff $(\lambda, \mu) = (\overline{\lambda}, \overline{\mu})$
\cite[Lemma~4]{Dolgopolik_GSP}. Arguing in the same way as above one can easily verify that
$\mathscr{L}(\overline{x}, \overline{\lambda}, \overline{\mu}, c) = f^*$ for all $c > 0$. Thus, one has
\begin{equation} \label{EAL_PartlyExact}
  \big\{ (x, \lambda, \mu) \in \Omega^* \times \Lambda \bigm| \eta(x, \lambda, \mu) = 0 \big\} 
  \subseteq \argmin_{(x, \lambda, \mu)} \mathscr{L}(x, \lambda, \mu, c) \quad \forall c \ge c_0.
\end{equation}
Let us prove the opposite inclusion. Then one can conclude that $\mathscr{L}(\xi, c)$ is globally extendedly exact.

Fix arbitrary $c > c_0$, and $\xi_0 = (x_0, \lambda_0, \mu_0) \in \argmin_{\xi} \mathscr{L}(\xi, c)$. Clearly, if 
$h(x_0) \ne 0$, then $\mathscr{L}(\xi_0, c) > \mathscr{L}(\xi_0, c_0) = f^*$, which is impossible due to
\eqref{EAL_GlobMinValue}. Consequently, $h(x_0) = 0$. Therefore from \eqref{EAL_Representation},
\eqref{EAL_InequalPenTerm_Represent} and \eqref{EAL_GlobMinValue}, the definition of $\xi_0$, and the fact that the
function $\mathscr{L}(\xi, c)$ is nondecreasing in $c$ it follows that
\begin{multline*}
  f^* = \mathscr{L}(\xi_0, c_0) = f(x_0) + \frac{1}{p(x_0, \lambda_0)} \Phi(x_0, \lambda_0, c_0) + \eta(\xi_0) \\
  \le f(x_0) + \frac{1}{p(x_0, \lambda_0)} 
  \left( - p(x_0, \lambda_0) \langle \lambda_0, y \rangle + \frac{c_0}{2} \| y \|_F^2 \right) + \eta(\xi_0)
\end{multline*}
for any $y \in \mathbb{S}^l_{-} - G(x_0)$. Hence for any $y \in (\mathbb{S}^l_{-} - G(x_0)) \setminus \{ 0 \}$ one has
$$
  f(x_0) + \frac{1}{p(x_0, \lambda_0)} 
  \left( - p(x_0, \lambda_0) \langle \lambda_0, y \rangle + \frac{c}{2} \| y \|_F^2 \right) + \eta(\xi_0) > f^*,
$$
which implies that the minimum in the definition of $\Phi(x_0, \lambda_0, c)$ (see
\eqref{EAL_InequalPenTerm_Represent}) is attained at the point $y = 0$, and $\Phi(x_0, \lambda_0, c) = 0$. Hence
$0 \in \mathbb{S}^l_{-} - G(x_0)$, i.e. $x_0$ is feasible, which yields that
$$
  f^* = \mathscr{L}(\xi_0, c) = f(x_0) + \eta(\xi_0) \ge f^*
$$
due to the fact that the function $\eta(\cdot)$ is nonnegative. Hence $\eta(\xi_0) = 0$ and $x_0 \in \Omega^*$. In
other words, the inclusion opposite to \eqref{EAL_PartlyExact} holds true, which completes the proof.
\end{proof}

\begin{remark} \label{Remark_EAL_ValueAtKKTpoints}
From the proof of the lemma above it follows that if $(x^*, \lambda^*, \mu^*)$ is a KKT-point of the problem
$(\mathcal{P})$, then $\mathscr{L}(x^*, \lambda^*, \mu^*, c) = f(x^*)$ for all $c > 0$.
\end{remark}

\begin{lemma} \label{Lemma_EAL_PenaltyType}
Let the functions $f$, $G$, and $h$ be continuously differentiable, and suppose that at least one of the globally
optimal solutions of the problem $(\mathcal{P})$ is nondegenerate. Then the augmented Lagrangian 
function $\mathscr{L}(x, \lambda, \mu, c)$ is a penalty-type separating function.
\end{lemma}

\begin{proof}
Let $\{ c_n \} \subset (0, + \infty)$ be an increasing unbounded sequence, and let also
$\xi_n = (x_n, \lambda_n, \mu_n) \in \argmin_{\xi} \mathscr{L}(\xi, c_n)$ for all $n \in \mathbb{N}$, and 
$\xi^* = (x^*, \lambda^*, \mu^*)$ be a cluster point of the sequence $\{ \xi_n \}$. Replacing, if necessary, the
sequence $\{ \xi_n \}$ with its subsequence one can suppose that $\xi_n$ converges to $\xi^*$.

Let $\overline{x}$ be a nondegenerate globally optimal solution of the problem $(\mathcal{P})$ that exists by our
assumption. Applying \cite[Proposition~4.75]{BonnansShapiro} one obtains that there exists 
$(\overline{\lambda}, \overline{\mu})$ such that the triplet 
$\overline{\xi} = (\overline{x}, \overline{\lambda}, \overline{\mu})$ is a KKT-point of the problem $(\mathcal{P})$.
Then $\mathscr{L}(\overline{\xi}, c) = f(\overline{x}) = f^*$ for all $c > 0$ by virtue of
Remark~\ref{Remark_EAL_ValueAtKKTpoints}, which implies that $\mathscr{L}(\xi_n, c_n) \le f^*$ for all 
$n \in \mathbb{N}$. 

Minimizing the function $\omega_n(t) = - \| \mu_n \| t + c_n t^2 / 2 q(x_n, \mu_n)$ with 
respect to $t \in \mathbb{R}$ (see~\eqref{EAL_SemiDef}) one obtains that
\begin{multline} \label{EAL_LowerEstimate}
  f^* \ge \mathscr{L}(\xi_n, c_n) \ge f(x_n) - \frac{\trace(\lambda_n^2) p(x_n, \lambda_n)}{2 c_n} 
  - \frac{\| \mu_n \|^2 q(x_n, \lambda_n)}{2 c_n} + \eta(\xi_n) \\
  \ge f(x_n) - \frac{\alpha}{c_n} + \eta(\xi_n)
\end{multline}
for all $n \in \mathbb{N}$. Hence passing to the limit as $n \to + \infty$ one gets that 
$f(x^*) + \eta(\xi^*) \le f^*$. Therefore it remains to prove that $x^*$ is feasible, since in this case one obtains
that $x^* \in \Omega^*$ and $\eta(\xi^*) = 0$, which implies the required result.

Arguing by reductio ad absurdum, suppose that $x^*$ is infeasible. Suppose, at first, that $h(x^*) \ne 0$. Then there
exist $\varepsilon > 0$ and a subsequence $\{ x_{n_k} \}$ such that $\| h(x_{n_k}) \| \ge \varepsilon$ for all 
$k \in \mathbb{N}$. Note that since $\{ \mu_n \}$ is a convergent sequence, there exists $M > 0$ such that
$\| \mu_n \| \le M$ for all $n \in \mathbb{N}$. Moreover, it is obvious that $\| h(x_n) \|^2 < \alpha$ for any 
$n \in \mathbb{N}$. Therefore
$$
  \mathscr{L}(\xi_{n_k}, c_{n_k}) \ge f(x_{n_k}) - \frac{\alpha}{2 c_{n_k}} - 
  M \sqrt{\alpha} + \frac{c_{n_k} \varepsilon^2}{2(\alpha - \varepsilon^2)}
$$
for all $k \in \mathbb{N}$ (clearly, one can suppose that $\varepsilon^2 < \alpha$). Hence one gets that
$\mathscr{L}(\xi_{n_k}, c_{n_k}) \to + \infty$ as $k \to + \infty$, which is impossible. Thus, $h(x^*) = 0$.

Suppose, now, that $G(x^*)$ is not negative semidefinite. Then there exist $\varepsilon > 0$ and a subsequence
$\{ x_{n_k} \}$ such that $\| [G(x_{n_k})]_+ \|_F^2 \ge \varepsilon$ for all $k \in \mathbb{N}$. From the facts that 
$\{ \lambda_n \}$ is a convergent sequence, and $c_n \to + \infty$ as $n \to \infty$ it follows that
$\| [G(x_{n_k}) + p(x_{n_k}, \lambda_{n_k}) \lambda_{n_k} / c_{n_k}]_+ \|_F^2 \ge \varepsilon / 2$ for any sufficiently
large $k \in \mathbb{N}$. Consequently, for any $k$ large enough one has
$$
  \mathscr{L}(\xi_{n_k}, c_{n_k}) \ge f(x_{n_k}) + \frac{c_{n_k} \varepsilon}{4 (\alpha - \varepsilon^{\varkappa})} - 
  \frac{\alpha}{c_{n_k}}
$$
(one can obviously suppose that $\varepsilon^{\varkappa} < \alpha$).
Hence $\mathscr{L}(\xi_{n_k}, c_{n_k}) \to + \infty$ as $k \to + \infty$, which is impossible. Thus, $x^*$ is a
feasible point of the problem $(\mathcal{P})$, and the proof is complete.
\end{proof}

\begin{lemma} \label{Lemma_SublevelSets}
Let the functions $f$, $G$, and $h$ be continuously differentiable. Suppose also that every globally optimal solution of
the problem $(\mathcal{P})$ is nondegenerate, and the set $\Omega(\alpha, \gamma)$ is bounded for some $\gamma > 0$.
Then the set $\{ \xi \in \mathbb{R}^d \times \Lambda \mid \mathscr{L}(\xi, c_0) < f^* \}$ is bounded for some $c_0 > 0$.
\end{lemma}

\begin{proof}
Arguing by reductio ad absurdum, suppose that the set 
$S(c) = \{ \xi \in \mathbb{R}^d \times \Lambda \mid \mathscr{L}(\xi, c) < f^* \}$ is unbounded for any $c > 0$. Then
for any $n \in \mathbb{N}$ there exists $\xi_n = (x_n, \lambda_n, \mu_n) \in S(n)$ such that
$\| x_n \| + \| \lambda_n \|_F + \| \mu_n \| \ge n$. Applying inequality \eqref{EAL_LowerEstimate} one obtains that
$x_n \in \Omega(\alpha, \gamma)$ for any sufficiently large $n$. Therefore, without loss of generality one can suppose
that the sequence $x_n$ converges to a point $x^*$. Let us verify that $x^*$ is a globally optimal solution of the
problem $(\mathcal{P})$.

For any $n \in \mathbb{N}$ denote $w_n = q(x_n, \mu_n) \langle \mu_n, h(x_n) \rangle + n \| h(x_n) \|^2 / 2$, and
$$
  u_n = \frac{1}{2n} \Big( \trace\big( [n G(x_n) + p(x_n, \lambda_n) \lambda_n]_+^2 \big) - 
  p(x_n, \lambda_n)^2 \trace(\lambda_n^2) \Big).
$$
It is easy to check that $u_n / p(x_n, \lambda_n) \ge - \alpha / 2 n$ and $w_n / q(x_n, \mu_n) \ge - \alpha / 2 n$
for all $n \in \mathbb{N}$, i.e. the sequences $\{ u_n / p(x_n, \lambda_n) \}$ and $\{ w_n / q(x_n, \mu_n \}$ are
bounded below.

From \eqref{EAL_LowerEstimate} and the boundedness of the set $\Omega(\alpha, \gamma)$ it follows that there exists
$\theta \in \mathbb{R}$ such that $\mathscr{L}(\xi, c) \ge \theta$ for any $\xi \in \mathbb{R}^d \times \Lambda$ and 
for all $c \ge 1$. On the other hand, from the definition of $\xi_n$ it follows that $\mathscr{L}(\xi_n, n) < f^*$.
Thus, the sequence $\{ \mathscr{L}(\xi_n, n) \}$ is bounded. 

Observe that 
$$
  \mathscr{L}(\xi_n, n) = f(x_n) + \frac{u_n}{p(x_n, \lambda_n)} + \frac{w_n}{q(x_n, \mu_n)} + \eta(\xi_n).
$$
The sequence $\{ f(x_n) \}$ is bounded due to the fact that $\{ x_n \}$ is a convergent sequence. Hence and from the
facts that the sequences $\{ u_n / p(x_n, \lambda_n) \}$ and $\{ w_n / q(x_n, \mu_n) \}$ are bounded below, and the
function $\eta(\cdot)$ is nonnegative it follows that the sequence $\{ \eta(\xi_n) \}$ is bounded as well. Therefore the
sequence $\{ u_n / p(x_n, \lambda_n) + w_n / q(x_n, \mu_n) \}$ is bounded, which with the use of the boundedness
below of the sequences $\{ u_n / p(x_n, \lambda_n) \}$ and $\{ w_n / q(x_n, \mu_n) \}$ implies that the sequences 
$\{ u_n / p(x_n, \lambda_n) \}$ and $\{ w_n / q(x_n, \mu_n) \}$ are bounded. By definition one has
$0 < p(x_n, \lambda_n) \le \alpha$ and $0 < q(x_n, \mu_n) \le \alpha$ for all $n \in \mathbb{N}$. Therefore the
sequences $\{ u_n \}$ and $\{ w_n \}$ are bounded. Hence arguing in a similar way to the proof of
Lemma~\ref{Lemma_EAL_PenaltyType} one can easily check that $\| [ G(x_n) ]_+ \|_F \to 0$ and $\| h(x_n) \| \to 0$ as 
$n \to \infty$, which implies that $x^*$ is a feasible point of the problem $(\mathcal{P})$.

Indeed, suppose, for instance, that $h(x^*) \ne 0$. Then there exist $\varepsilon > 0$ and a subsequence 
$\{ x_{n_k} \}$ such that $\| h(x_{n_k} \| > \varepsilon$ for all $k \in \mathbb{N}$. Therefore 
$w_{n_k} \ge - \alpha \varepsilon + n_k \varepsilon^2 / 2$ for all $k \in \mathbb{N}$. Hence 
$w_{n_k} \to + \infty$ as $k \to \infty$, which contradicts the boundedness of the sequence $\{ w_n \}$. Thus,
$h(x^*) = 0$. The validity of the relation $G(x^*) \preceq 0$ is proved in a similar way.

From inequality \eqref{EAL_LowerEstimate} and the fact that $\xi_n \in S(n)$ it follows that 
$f(x_n) < f^* + \alpha / n$. Hence passing to the limit as $n \to \infty$ one obtains that $f(x^*) \le f^*$, which
implies that $x^* \in \Omega^*$ due to the feasibility of $x^*$.

By our assumption, the point $x^*$ is nondegerate. Hence by \cite[Proposition~4.75]{BonnansShapiro} there exists
$(\lambda^*, \mu^*)$ such that the triplet $\xi^* = (x^*, \lambda^*, \mu^*)$ is a KKT-point of the problem
$(\mathcal{P})$. Therefore the matrix $D^2_{\nu \nu} \eta(\xi^*)$, where $\nu = (\lambda, \mu)$, is positive definite by
\cite[Lemma~4]{Dolgopolik_GSP}. Consequently, taking into account the fact that $\nu \to \eta(x, \nu)$ is
obviously a quadratic function whose Hessian matrix continuously depends on $x$, one can easily verify that there
exist $\delta_1 > 0$, $\delta_2 \in \mathbb{R}$, $\delta_3 \in \mathbb{R}$ and a neighbourhood $U$ of $x^*$ such that 
$$
  \eta(x, \lambda, \mu) \ge \delta \big( \| \lambda \|_F^2 + \| \mu \|^2 \big)
  + \delta_2 \big( \| \lambda \|_F + \| \mu \| \big) + \delta_3 
  \quad \forall x \in U \: \forall (\lambda, \mu) \in \Lambda.
$$
Recall that $\mathscr{L}(\xi_n, n) \ge f(x_n) - \alpha / n + \eta(\xi_n)$ for all $n \in \mathbb{N}$ by virtue of
inequality \eqref{EAL_LowerEstimate}. Applying the above lower estimate for the term $\eta(\xi)$, and the facts
that $x_n$ converges to $x^*$, and $\| x_n \| + \| \lambda_n \|_F + \| \mu_n \| \ge n$ for all $n \in \mathbb{N}$
one obtains that $\mathscr{L}(\xi_n, n) \to + \infty$ as $n \to \infty$, which is impossible. Thus, the set $S(c)$ is
bounded for some $c > 0$.
\end{proof}

Now, applying Lemmas~\ref{Lemma_EAL_IntersectImpliesExact}--\ref{Lemma_SublevelSets} and the localization principle in
the extended form (Theorems~\ref{Thrm_LocPrincipleExtended} and \ref{Thrm_LocPrincipleExtended_SubLevel}) one obtains
that Theorem~\ref{Thrm_EAL_SemiDef_GlobalExact} holds true.

Let us also obtain simple sufficient conditions for the local extended exactness of the augmented Lagrangian function
$\mathscr{L}(x, \lambda, \mu, c)$. To this end, let us recall second order sufficient optimality conditions for
nonlinear semidefinite programming problems.

Let $x^*$ be a locally optimal solution of the problem $(\mathcal{P})$, and suppose that $x^*$ is nondegenerate. Then
by \cite[Proposition~4.75]{BonnansShapiro} there exists a unique pair $(\lambda^*, \mu^*) \in \Lambda$ such that the
triplet $(x^*, \lambda^*, \mu^*)$ is a KKT-point of the problem $(\mathcal{P})$. Suppose, at first, that 
$\rank(G(x^*)) < l$. We say that \textit{the second order sufficient optimality condition} holds true at $x^*$
(see~\cite[Theorem~5.89]{BonnansShapiro}) iff the matrix
\begin{equation} \label{SemiDef_SecondOrderOptCond}
  \nabla^2_{xx} L(x^*, \lambda^*, \mu^*) - 
  2 \Big[ \trace\Big( \lambda^* \cdot 
  \big( D_{x_i} G(x^*) G(x^*)^{\dagger} D_{x_j} G(x^*) \big)  \Big) \Big]_{i, j = 1}^d
\end{equation}
is positive definite on \textit{the critical cone}
\begin{equation} \label{SemiDef_CriticalCone}
  \Big\{ v \in \mathbb{R}^d \Bigm| \sum_{i = 1}^d v_i E_0^T D_{x_i} G(x^*) E_0 \preceq 0, \:
  \nabla h(x^*) v = 0, \: \langle \nabla f(x^*), v \rangle = 0 \Big\},
\end{equation}
where $G(x^*)^{\dagger}$ is the Moore-Penrose pseudoinverse of the matrix $G(x^*)$, and $E_0$ is 
$l \times (l - \rank(G(x^*)))$ matrix composed from the eigenvectors of $G(x^*)$ corresponding to its zero eigenvalue.
Note that if $\rank(G(x^*)) = l$, then the constraint $G(x) \preceq 0$ is inactive at $x^*$, i.e. $x^*$ is a locally
optimal solution of the problem of minimizing $f(x)$ subject to $h(x) = 0$. In this case $\lambda^* = 0$, and we set
$E_0 = 0$.

In the proof of the theorem below we combine the second-order analysis of Rockafellar-Wets' augmented Lagrangian from
\cite{ShapiroSun} (see~Section~3.1 of this paper) with the proof of Theorem~8 from \cite{Dolgopolik_GSP}.

\begin{theorem}
Let $\varkappa > 1$, and the functions $f$, $G$ and $h$ be twice differentiable at a locally optimal solution $x^*$ of
the problem $(\mathcal{P})$. Suppose also that $x^*$ is nondegenerate, and the second order sufficient optimality
condition holds true at $x^*$. Then $\mathscr{L}(x, \lambda, \mu, c)$ is locally extendedly exact at $x^*$.
\end{theorem}

\begin{proof}
Let $(\lambda^*, \mu^*) \in \Lambda$ be a unique pair such that the triplet $(x^*, \lambda^*, \mu^*)$ is a KKT-point of
the problem $(\mathcal{P})$ that exists due to the fact that $x^*$ is nondegenerate. Note that by
\cite[Lemma~4]{Dolgopolik_GSP} one has $\eta(x^*, \lambda, \mu) = 0$ iff $\lambda = \lambda^*$ and $\mu = \mu^*$. Hence
taking into account the fact that the function $\mathscr{L}(x, \lambda, \mu, c)$ is nondecreasing in $c$ one obtains
that $\mathscr{L}(x, \lambda, \mu, c)$ is locally extendedly exact at $x^*$ iff $\xi^* = (x^*, \lambda^*, \mu^*)$ is a
point of local minimum of this function for any sufficiently large $c > 0$.

In order to verify that $\xi^*$ is indeed a point of local minimum of $\mathscr{L}(x, \lambda, \mu, c)$, let us
compute a second order expansion of this function in a neighbourhood of $\xi^*$. We start by computing the gradient of
this function at $\xi^*$. To this end, denote $K =  \mathbb{S}^l_{-}$, $\delta(y) = \dist^2 (y, K)$ for any 
$y \in \mathbb{S}^l$, and
$$
  \Phi_0(x, \lambda, c) = \delta\Big( G(x) + \frac{p(x, \lambda)}{c} \lambda \Big).
$$
Then for any $\xi = (x, \lambda, \mu) \in \Omega_{\alpha} \times \Lambda$ one has
\begin{multline} \label{EAL_SDP_Represendation}
  \mathscr{L}(\xi, c) = f(x) 
  + \frac{c}{2 p(x, \lambda)} \Big( \Phi_0(x, \lambda, c) - \frac{p(x, \lambda)^2}{c^2} \trace(\lambda^2) \Big) \\
  + \langle \mu, h(x) \rangle + \frac{c}{2 q(x, \mu)} \| h(x) \|^2 + \eta(\xi)
\end{multline}
(see the proof of Lemma~\ref{Lemma_EAL_IntersectImpliesExact}). By \cite[Theorem~4.13]{BonnansShapiro} the function
$\delta(y)$ is Fr\'echet differentiable, and $D \delta(y) = 2 (y - P_K(y))$ for all $y \in \mathbb{S}^l$. As it was
noted in the proof of Lemma~\ref{Lemma_EAL_IntersectImpliesExact}, one has 
$P_K (G(x^*) + p(x^*, \lambda^*) \lambda^* / c) = G(x^*)$ by virtue of the fact that $(x^*, \lambda^*, \mu^*)$ is a
KKT-point, which, in particular, implies that 
\begin{equation} \label{SDP_InequalTermValue}
  \Phi_0(x^*, \lambda^*, c) = \frac{p(x^*, \lambda^*)^2}{c^2} \| \lambda^* \|_F^2. 
\end{equation}
Observe also that since $p(x, \lambda) = (\alpha - \delta(G(x))^{\varkappa}) / (1 + \| \lambda \|^2_F)$ and 
$G(x^*) \in K$, one has $\nabla_x p(x^*, \lambda^*) = 0$. Therefore applying the chain rule one obtains that
\begin{equation} \label{SDP_InequalTermGradient}
\begin{split}
  \nabla_x \Phi_0(x^*, \lambda^*, c) & = 
  \frac{2 p(x^*, \lambda^*)}{c} \nabla_x \trace( \lambda^* G(x^*) ), \\
  \nabla_{\lambda} \Phi_0(x^*, \lambda^*, c) & = \frac{2 p(x^*, \lambda^*)^2}{c^2} \lambda^* + 
  \frac{2 p(x^*, \lambda^*)}{c^2} \| \lambda^* \|_F^2 \nabla_{\lambda} p(x^*, \lambda^*).
\end{split}
\end{equation}
Clearly, the function $\eta$ is differentiable at $\xi^*$. Hence and from the facts that the function $\eta(\xi)$ is
nonnegative, and $\eta(\xi^*) = 0$ one gets that 
\begin{equation} \label{SDP_PenTermGradient}
  \nabla_x \eta(\xi^*) = 0, \quad \nabla_{\lambda} \eta(\xi^*) = 0, \quad \nabla_{\mu} \eta(\xi^*) = 0.
\end{equation}
Consequently, computing the gradient of augmented Lagrangian \eqref{EAL_SDP_Represendation} with the use of
\eqref{SDP_InequalTermValue}, \eqref{SDP_InequalTermGradient} and \eqref{SDP_PenTermGradient} one obtains that
$$
  \nabla_{x} \mathscr{L}(\xi^*, c) = \nabla_x L(\xi^*) = 0, \quad
  \nabla_{\lambda} \mathscr{L}(\xi^*, c) = 0, \quad \nabla_{\mu} \mathscr{L}(\xi^*, c) = 0.
$$
(recall that $\xi^*$ is a KKT-point).

Now, let us compute a second order expansion of $\mathscr{L}(\xi, c)$ at $\xi^*$. At first, note that from the facts
that the functions $f$, $G$ and $h$ are twice differentiable at $x^*$ and $\nabla_x L(\xi^*) = 0$ it follows that for
any $\xi \in \mathbb{R}^d \times \Lambda$ one has
$$
  \nabla_x L(\xi, c) = D_{\xi} \Big( \nabla_x L(\xi^*) \Big) (\xi - \xi^*) + o( \| \xi - \xi^* \| ).
$$
Therefore
$$
  \big\| \nabla_x L(\xi, c) \big\|^2 =
  \Big\| D_{\xi} \Big( \nabla_x L(\xi^*) \Big) (\xi - \xi^*) \Big\|^2 + o( \| \xi - \xi^* \|^2 ),
$$
which implies that the function $\eta(\xi)$ is twice differentiable at $\xi^*$. Furthermore, the matrix 
$\nabla^2_{\nu \nu} \eta(\xi^*)$, where $\nu = (\lambda, \mu)$, is positive definite by \cite[Lemma~4]{Dolgopolik_GSP}.

Applying \cite[Theorem~4.133]{BonnansShapiro} (see also \cite[Sect.~3.1]{ShapiroSun}) one obtains that for all 
$y, v \in \mathbb{S}^l$ there exists the second-order Hadamard directional derivative
$$
  \delta''(y; v) := \lim_{[v', t] \to [v, +0]} 
  \frac{\delta(y + t v') - \delta(y) - t D \delta(y) v'}{\frac{1}{2}t^2},
$$
and
\begin{equation} \label{Dist_2ndHadamardDeriv_Representation}
  \delta''(y; v) = \min_{z \in \mathscr{C}(y)} 
  \Big[ 2 \| v - z \|^2 - 2 \sigma\big( y - \overline{y}, T^2_{K}(\overline{y}, z) \big) \Big],
\end{equation}
where $\mathscr{C}(y) = \{ z \in T_K(\overline{y}) \mid \langle y - \overline{y}, z \rangle = 0 \}$, $\overline{y}$ is
the projection of $y$ onto $K$, 
\begin{equation} \label{2ndOrderTangeSet_SupportFunc}
  \sigma( y - \overline{y}, T^2_{K}(\overline{y}, z)) = 
  \sup_{h \in T^2_{K}(\overline{y}, z)} \langle y - \overline{y}, h \rangle,
\end{equation}
and $T^2_{K}(\overline{y}, z)$ is \textit{the second-order tangent set} to the cone $K$ at the
point $\overline{y}$ in direction $z$ (see~\cite[Section~3.2.1]{BonnansShapiro}). Let us note that the restriction of
the support function \eqref{2ndOrderTangeSet_SupportFunc} to the cone $\mathscr{C}(y)$ is a non-positive quadratic
function of $z$ (see \cite{BonnansShapiro}, formulae~(5.205), (5.209) and (5.210)). Therefore, as one can easily verify,
the function $\delta''(y; \cdot)$ is finite, continuous and positively homogeneous of degree two. Utilizing these
properties of the function $\delta''(y, \cdot)$ one can check that for any bounded linear operator 
$T \colon E \to \mathbb{S}^l$ (here $E$ is a finite dimensional normed space) one has
\begin{equation} \label{Dist_2ndOrderExpansion}
  \delta\Big( y + T w + o(\| w \|) \Big) = \delta(y) + 
  D \delta(y) \Big( T w + o(\| w \|) \Big) + \frac{1}{2} \delta''\big(y; T w \big) + o( \| w \|^2 ).
\end{equation}
Hence and from the fact that for any $y \in K$ one has $\delta(y) = 0$ and $D \delta(y) = 2(y - P_K(y)) = 0$ it follows
that the function $\delta(\cdot)^{\varkappa}$ is twice differentiable at the point $G(x^*)$ and 
$D^2 \delta^{\varkappa}(G(x^*)) = 0$ (recall that $\varkappa > 1$). Therefore the function $p(x, \lambda)$ is twice
differentiable at $(x^*, \lambda^*)$ and
\begin{equation} \label{BarrierTermDerivatives}
  \nabla_x p(x^*, \lambda^*) = 0, \quad \nabla_{x \lambda} p(x^*, \lambda^*) = 0, \quad
  \nabla_{xx} p(x^*, \lambda^*) = 0.
\end{equation}
Consequently, for any $w = (w_x, w_{\lambda}) \in \mathbb{R}^d \times \Lambda$ in a neighbourhood of zero one has
\begin{multline*}
  G(x^* + w_x) + \frac{p(x^* + w_x, \lambda + w_{\lambda})}{c} (\lambda^* + w_{\lambda}) 
  = G(x^*) + \frac{p(x^*, \lambda^*)}{c} \lambda^* \\
  + D G(x^*) w_x + \frac{p(x^*, \lambda^*)}{c} w_{\lambda} 
  + \frac{1}{c} \langle \nabla_{\lambda} p(x^*, \lambda^*), w_{\lambda} \rangle \lambda^* 
  + \frac{1}{2} D^2 G(x^*)(w_x, w_x) \\
  + \frac{1}{2c} \langle \nabla_{\lambda} p(x^*, \lambda^*), w_{\lambda} \rangle w_{\lambda}
  + \frac{1}{2 c} D^2_{\lambda \lambda}p(x^*, \lambda^*)(w_{\lambda}, w_{\lambda}) \lambda^*
  + o( \| w \|^2 ),
\end{multline*}
Hence denoting the linear terms in this expansion by $T_c w$, i.e.
$$
  T_c w = D G(x^*) w_x + \frac{p(x^*, \lambda^*)}{c} w_{\lambda} 
  + \frac{1}{c} \langle \nabla_{\lambda} p(x^*, \lambda^*), w_{\lambda} \rangle \lambda^*,
$$
and applying \eqref{Dist_2ndOrderExpansion} one obtains that
\begin{multline*}
  \Phi_0(x^* + w_x, \lambda^* + w_{\lambda}, c) - \Phi_0(x^*, \lambda^*, c)
  = \frac{2 p(x^*, \lambda^*)}{c} \big\langle \lambda^*, T_c w \big\rangle \\
  + \frac{p(x^*, \lambda^*)}{c} \left\langle \lambda^*, D^2 G(x^*)(w_x, w_x) + 
  \frac{1}{c} D^2_{\lambda \lambda}p(x^*, \lambda^*)(w_{\lambda}, w_{\lambda}) \lambda^* \right\rangle \\
  + \frac{p(x^*, \lambda^*)}{c^2} \big\langle \nabla_{\lambda} p(x^*, \lambda^*), w_{\lambda} \big\rangle \cdot
  \big\langle \lambda^*, w_{\lambda} \big\rangle \\
  + \frac{1}{2} \delta''\left( G(x^*) + \frac{p(x^*, \lambda^*)}{c} \lambda^*; T_c w \right)
  + o( \| w \|^2 )
\end{multline*}
(recall that $D \delta(y) = 2 (y - P_K(y))$ and $P_K (G(x^*) + p(x^*, \lambda^*) \lambda^* / c) = G(x^*)$). Therefore
taking into account \eqref{BarrierTermDerivatives}, \eqref{SDP_InequalTermValue} and \eqref{SDP_PenTermGradient} one
obtains that for any $c > 0$ the augmented Lagrangian function $\mathscr{L}(\xi, c)$ admits the following second order
expansion at the point $\xi^*$:
\begin{multline*}
  \mathscr{L}(\xi, c) - \mathscr{L}(\xi^*, c) = 
  \frac{1}{2} \big\langle x - x^*, \nabla_{xx}^2 L(\xi^*) (x - x^*) \big\rangle
  + \frac{1}{2} \varphi(\xi - \xi^*, c) \\
  + \frac{1}{2} D^2 \eta(\xi^*) (\xi - \xi^*, \xi - \xi^*) + o\big( \| \xi- \xi^* \|^2 \big).
\end{multline*}
Here
\begin{multline*}
  \varphi(\xi, c) = \frac{c}{2 p(x^*, \lambda^*)} 
  \delta''\left( G(x^*) + \frac{p(x^*, \lambda^*)}{c} \lambda^*; T_c (x, \lambda) \right) \\
  - \frac{2}{p(x^*, \lambda^*)} \big\langle \nabla_{\lambda} p(x^*, \lambda^*) \lambda \big\rangle \cdot
  \big\langle \lambda^*, D G(x^*) x \big\rangle
  - \frac{1}{c} \big\langle \lambda^*, \lambda \big\rangle \cdot 
  \big\langle \nabla_{\lambda} p(x^*, \lambda^*), \lambda \big\rangle \\
  - \frac{1}{p(x^*, \lambda^*) c} \| \lambda^* \|_F^2 
  \big( \langle \nabla_{\lambda} p(x^*, \lambda^*), \lambda \rangle \big)^2 
  - \frac{p(x^*, \lambda^*)}{c} \| \lambda \|^2 \\
  + \langle \mu, \nabla h(x^*) x \rangle
  + \frac{c}{q(x^*, \lambda^*)}  \big\| \nabla h(x^*) x \big\|^2.
\end{multline*}
From \eqref{Dist_2ndHadamardDeriv_Representation} it follows that
\begin{multline*}
  \frac{c}{2 p(x^*, \lambda^*)} 
  \delta''\left( G(x^*) + \frac{p(x^*, \lambda^*)}{c} \lambda^*; T_c (x, \lambda) \right) \\
  = \min_{z \in C_0(x^*, \lambda^*)} \Big\{ \Big\| \frac{c}{p(x^*, \lambda^*)} \big(D G(x^*) x - z \big) + 
  \lambda + 
  \frac{1}{p(x^*, \lambda^*)} \langle \nabla_{\lambda} p(x^*, \lambda^*), \lambda \rangle \lambda^* \Big\|^2 \\
  - \sigma\big( \lambda^*, T^2_K(G(x^*), z) \big) \Big\},
\end{multline*}
where 
$$
  C_0(x^*, \lambda^*) = \mathscr{C}\left( G(x^*) + \frac{p(x^*, \lambda^*)}{c} \lambda^*\right)
  = \{ z \in T_K(G(x^*)) \mid \langle \lambda^*, z \rangle = 0 \}.
$$
As it was noted above, the function $\sigma( \lambda^*, T^2_K(G(x^*), \cdot))$ is quadratic and nonpositive on the cone 
$C_0(x^*, \lambda^*)$. Consequently, for any $\xi_0 = (x_0, \lambda_0) \in \mathbb{R}^d \times \Lambda$ one has
\begin{equation} \label{SecondOrderTermLim_1}
  \limsup_{[\xi, c] \to [\xi_0, + \infty]} \varphi(\xi, c) \ge 
  - \sigma\big( \lambda^*, T^2_K(G(x^*), D G(x^*) x_0) \big) \ge 0,
\end{equation}
if $D G(x^*) x_0 \in C_0(x^*, \lambda^*)$ and $\nabla h(x^*) x_0 = 0$, and
\begin{equation} \label{SecondOrderTermLim_2}
  \lim_{[\xi, c] \to [\xi_0, + \infty]} \varphi(\xi, c) = + \infty,
\end{equation}
otherwise.

Now we can turn to the proof of the local exactness. Utilizing the second order expansion for the augmented Lagrangian
$\mathscr{L}(\xi, c)$ one obtains that for any $c > 0$ there exists a neighbourhood $U_c$ of $\xi^*$ such that 
for all $\xi \in U_c$ one has
\begin{multline*} 
  \Big| \mathscr{L}(\xi, c) - \mathscr{L}(\xi^*, c) -
  \frac{1}{2} \big\langle x - x^*, \nabla_{xx}^2 L(\xi^*) (x - x^*) \big\rangle
  - \frac{1}{2} \varphi(\xi - \xi^*, c) \\
  - \frac{1}{2} D^2 \eta(\xi^*) (\xi - \xi^*, \xi - \xi^*) \Big| < \frac{1}{c} \| \xi- \xi^* \|^2.
\end{multline*}
Arguing by reductio ad absurdum, suppose that $\mathscr{L}(\xi, c)$ is not locally extendedly exact at $\xi^*$. Then
for any $n \in \mathbb{N}$ there exists $\xi_n \in U_n$ such that $\mathscr{L}(\xi_n, n) < \mathscr{L}(\xi^*, n)$.
Therefore
\begin{multline} \label{EAL_SecondOrderExpnasion}
  0 > \big\langle x_n - x^*, \nabla_{xx}^2 L(\xi^*) (x_n - x^*) \big\rangle + \varphi(\xi_n - \xi^*, n) \\
  + D^2 \eta(\xi^*) (\xi_n - \xi^*, \xi_n - \xi^*) - \frac{2}{n} \| \xi_n - \xi^* \|^2.
\end{multline}
for any $n \in \mathbb{N}$. Denote $\alpha_n = \| \xi_n - \xi^* \|$, $w_n = (x_n - x^*) / \alpha_n$, and 
$\nu_n = (\lambda_n, \mu_n) / \alpha_n$. Without loss of generality one can suppose that the sequence 
$\{ (w_n, \nu_n) \}$ converges to a point $\{ (w^*, \nu^*) \}$ such that $\| (w^*, \nu^*) \| = 1$.

Let us check that $w^* \ne 0$. Indeed, arguing by reductio ad absurdum, suppose that $w^* = 0$. Then dividing
inequality \eqref{EAL_SecondOrderExpnasion} by $\alpha_n^2$, and passing to the limit superior as $n \to \infty$ with
the use of \eqref{SecondOrderTermLim_1}, \eqref{SecondOrderTermLim_2}, and the fact that the function 
$\varphi(\cdot, c)$ is positively homogeneous of degree two (as well as $\delta''(y; \cdot)$) one obtains that 
$D^2_{\nu \nu} \eta(\xi^*)(\nu^*, \nu^*) \le 0$, which is impossible due to the fact that the function 
$D^2_{\nu \nu} \eta(\xi^*)(\cdot, \cdot)$ is positive definite by \cite[Lemma~4]{Dolgopolik_GSP}. Thus, $w^* \ne 0$.

Observe that $D^2 \eta(\xi^*)(\xi_n - \xi^*, \xi_n - \xi^*) \ge 0$ due to the facts that $\eta(\xi^*) = 0$, and 
the function $\eta$ is nonnegative. Hence and from \eqref{EAL_SecondOrderExpnasion} it follows that
$$
  0 > \big\langle w_n, \nabla_{xx}^2 L(\xi^*) w_n \big\rangle + 
  \varphi\left(\frac{1}{\alpha_n} (\xi_n - \xi^*), n \right) - \frac{2}{n}.
$$
Therefore passing to the limit superior as $n \to + \infty$ with the use of \eqref{SecondOrderTermLim_1} and
\eqref{SecondOrderTermLim_2} one obtains that $D G(x^*) w^* \in C_0(x^*, \lambda^*)$, $\nabla h(x^*) w^* = 0$, and
$$
  0 \ge \big\langle w^*, \nabla_{xx}^2 L(\xi^*) w^* \big\rangle 
  - \sigma\big( \lambda^*, T^2_K(G(x^*), D G(x^*) w^*) \big).
$$
Consequently, utilizing the known representation of the above ``sigma-term'' 
(see~\cite{BonnansShapiro}, formulae (5.208) and (5.209)) one obtains that 
$0 \ge \langle w^*, \Theta(x^*, \lambda^*) w^* \rangle$, where
$$
  \Theta(x^*, \lambda^*) =  \nabla^2_{xx} L(x^*, \lambda^*, \mu^*) - 
  2 \Big[ \trace\Big( \lambda^* \cdot 
  \big( D_{x_i} G(x^*) G(x^*)^{\dagger} D_{x_j} G(x^*) \big)  \Big) \Big]_{i, j = 1}^d.
$$
(cf.~\eqref{SemiDef_SecondOrderOptCond}). Furthermore, taking into accout the facts that 
$D G(x^*) w^* \in C_0(x^*, \lambda^*)$, $\nabla h(x^*) w^* = 0$ and $\nabla_x L(\xi^*) = 0$ (recall that $\xi^*$ is a
KKT-point) one obtains that $D G(x^*) w^* \in T_K(G(x^*))$ and $\langle \nabla f(x^*), w^* \rangle = 0$. Hence with the
use of the description of the contingent cone $T_K(G(x^*))$ in terms of the mapping $G(x^*)$
\cite[Example~2.65]{BonnansShapiro} one gets that $w^*$ belongs to the critical cone \eqref{SemiDef_CriticalCone},
which contradicts the fact that the second order sufficient optimality condition holds true at $x^*$. Thus, the
augmented Lagrangian function $\mathscr{L}(\xi, c)$ is locally exact at $\xi^*$.
\end{proof}

\section{Conclusions}

In this paper we developed a general theory of globally extendedly exact separating functions for constrained
optimization problems in finite dimensional spaces. We utilized this theory in order to obtain, in a simple and
straightforward manner, necessary and sufficient conditions for the global exactness of Huyer and Neumaier's extended
penalty function, and to design a new globally exact continuously differentiable augmented Lagrangian function for
nonlinear semidefinite programming problems.

\bibliographystyle{abbrv}  
\bibliography{ESF_II_bibl}

\begin{thebibliography}{10}

\bibitem{Bingzhuang}
L.~Bingzhuang and Z.~Wenling.
\newblock A modified exact smooth penalty function for nonlinear constrained
  optimization.
\newblock {\em J. Inequal. Appl.}, 1, 2012.

\bibitem{BonnansShapiro}
J.~F. Bonnans and A.~Shapiro.
\newblock {\em Perturbation analysis of optimization problems}.
\newblock Springer Science+Business Media, New York, 2000.

\bibitem{DiPilloGrippo1979}
G.~{\relax Di Pillo} and L.~Grippo.
\newblock A new class of augmented {L}agrangians in nonlinear programming.
\newblock {\em SIAM J. Control Optim.}, 17:618--628, 1979.

\bibitem{DiPilloGrippo1982}
G.~{\relax Di Pillo} and L.~Grippo.
\newblock A new augmented {L}agrangian function for inequality constraints in
  nonlinear programming problems.
\newblock {\em J. Optim. Theory Appl.}, 36:495--519, 1982.

\bibitem{DiPilloGrippo1980}
G.~{\relax Di Pillo}, L.~Grippo, and F.~Lampariello.
\newblock A method for solving equality constrained optimization problems by
  unconstrained minimization.
\newblock In K.~Iracki, K.~Malanowski, and S.~Walukiewicz, editors, {\em
  Optimization techniques: proceedings of the 9th IFIP Conference on
  Optimization Techniques}, pages 96--105. Springer-Verlag, Berlin, Heidelberg,
  1980.

\bibitem{DiPilloLiuzzi2003}
G.~{\relax Di Pillo}, G.~Liuzzi, S.~Lucidi, and L.~Palagi.
\newblock An exact augmented {L}agrangian function for nonlinear programming
  with two-sided constraints.
\newblock {\em Comput. Optim. Appl.}, 25:57--83, 2003.

\bibitem{DiPilloEtAl2002}
G.~{\relax Di Pillo}, G.~Liuzzi, S.~Lucidi, and L.~Palagi.
\newblock Fruitful uses of smooth exact merit functions in constrained
  optimization.
\newblock In G.~{\relax Di Pillo} and A.~Murli, editors, {\em High Performance
  Algorithms and Software for Nonlinear Optimization}, pages 201--225. Kluwer
  Academic Publishers, Dordrecht, 2003.

\bibitem{DiPilloLucidi1996}
G.~{\relax Di Pillo} and S.~Lucidi.
\newblock On exact augmented {L}agrangian functions in nonlinear programming.
\newblock In G.~{\relax Di Pillo} and F.~Giannessi, editors, {\em Nonlinear
  Optimization and Applications}, pages 85--100. Plenum Press, New York, 1996.

\bibitem{DiPilloLucidi2001}
G.~{\relax Di Pillo} and S.~Lucidi.
\newblock An augmented {L}agrangian function with improved exactness
  properties.
\newblock {\em SIAM J. Optim.}, 12:376--406, 2001.

\bibitem{DiPilloLucidiPalagi1993}
G.~{\relax Di Pillo}, S.~Lucidi, and L.~Palagi.
\newblock An exact penalty-{L}agrangian approach for a class of constrained
  optimization problems with bounded variables.
\newblock {\em Optim.}, 28:129--148, 1993.

\bibitem{DiPilloLiuzzi2011}
G.~{\relax Di Pillo}, G.~Luizzi, and S.~Lucidi.
\newblock An exact penalty-{L}agrangian approach for large-scale nonlinear
  programming.
\newblock {\em Optim.}, 60:223--252, 2011.

\bibitem{Dolgopolik_OptLet2}
M.~V. Dolgopolik.
\newblock Smooth exact penalty function {II}: a reduction to standard exact
  penalty functions.
\newblock {\em Optim. Lett.}, 10:1541--1560, 2016.

\bibitem{Dolgopolik_OptLet}
M.~V. Dolgopolik.
\newblock Smooth exact penalty functions: a general approach.
\newblock {\em Optim. Lett.}, 10:635--648, 2016.

\bibitem{Dolgopolik_UT}
M.~V. Dolgopolik.
\newblock A unifying theory of exactness of linear penalty functions.
\newblock {\em Optim.}, 65:1167--1202, 2016.

\bibitem{Dolgopolik_UnifiedApproach_I}
M.~V. Dolgopolik.
\newblock A unified approach to the global exactness of penalty and augmented
  {L}agrangian functions {I}: parametric exactness.
\newblock {\em arXiv: 1709.07073}, 2017.

\bibitem{DolgopolikMV_UT_2}
M.~V. Dolgopolik.
\newblock A unifying theory of exactness of linear penalty functions {II}:
  parametric penalty functions.
\newblock {\em Optim.}, 66:1577--1622, 2017.

\bibitem{Dolgopolik_GSP}
M.~V. Dolgopolik.
\newblock Augmented {L}agrangian functions for cone constrained optimization:
  the existence of global saddle points and exact penalty property.
\newblock {\em J. Glob. Optim.}, 71:237--296, 2018.

\bibitem{DuLiangZhang2006}
X.~Du, Y.~Liang, and L.~Zhang.
\newblock Further study on a class of augmented {L}agrangians of {\relax di
  pillo} and grippo in nonlinear programming.
\newblock {\em J. Shanghai Univ. (Engl. Ed.)}, 10:293--298, 2006.

\bibitem{DuZhangGao2006}
X.~Du, L.~Zhang, and Y.~Gao.
\newblock A class of augmented {L}agrangians for equality constraints in
  nonlinear programming problems.
\newblock {\em Appl. Math. Comput.}, 172:644--663, 2006.

\bibitem{FukudaLourenco}
E.~Fukuda and B.~F. Lourenco.
\newblock Exact augmented lagrangian functions for nonlinear semidefinite
  programming.
\newblock {\em Comput. Optim. Appl.}, 71:457--482, 2018.

\bibitem{Giannessi_book}
F.~Giannessi.
\newblock {\em Constrained Optimization and Image Space Analysis. Volume 1:
  Separation of Sets and Optimality Conditions}.
\newblock Springer, New York, 2005.

\bibitem{HuyerNeumaier}
W.~Huyer and A.~Neumaier.
\newblock A new exact penalty function.
\newblock {\em SIAM J. Optim.}, 13:1141--1158, 2003.

\bibitem{JianLin}
C.~Jiang, Q.~Lin, C.~Yu, K.~L. Teo, and G.-R. Duan.
\newblock An exact penalty method for free terminal time optimal control
  problem with continuous inequality constraints.
\newblock {\em J. Optim. Theory Appl.}, 154:30--53, 2012.

\bibitem{LiYu}
B.~Li, C.~J. Yu, K.~L. Teo, and G.~R. Duan.
\newblock An exact penalty function method for continuous inequality
  constrained optimal control problem.
\newblock {\em J. Optim. Theory Appl.}, 151:260--291, 2011.

\bibitem{LinLoxton}
Q.~Lin, R.~Loxton, K.~L. Teo, and Y.~H. Wu.
\newblock Optimal feedback control for dynamic systems with state constraints:
  An exact penalty approach.
\newblock {\em Optim. Lett.}, 8:1535--1551, 2014.

\bibitem{LinWuYu}
Q.~Lin, R.~Loxton, K.~L. Teo, Y.~H. Wu, and C.~Yu.
\newblock A new exact penalty method for semi-infinite programming problems.
\newblock {\em J. Comput. Appl. Math.}, 261:271--286, 2014.

\bibitem{Lucidi1988}
S.~Lucidi.
\newblock New results on a class of exact augmented {L}agrangians.
\newblock {\em J. Optim. Theory Appl.}, 58:259--282, 1988.

\bibitem{LuoWuLiu2013}
H.~Luo, H.~Wu, and J.~Liu.
\newblock Some results on augmented {L}agrangians in constrained global
  optimization via image space analysis.
\newblock {\em J. Optim. Theory Appl.}, 159:360--385, 2013.

\bibitem{LuoWuChen2012}
H.~Z. Luo, H.~X. Wu, and G.~T. Chen.
\newblock On the convergence of augmented {L}agrangian methods for nonlinear
  semidefinite programming.
\newblock {\em J. Glob. Optim.}, 54:599--618, 2012.

\bibitem{MaLiYiu}
C.~Ma, X.~Li, K.-F.~C. Yiu, and L.-S. Zhang.
\newblock New exact penalty function for solving constrained finite min-max
  problems.
\newblock {\em Appl. Math. Mech.-Engl. Ed.}, 33:253--270, 2012.

\bibitem{MaZhang2015}
C.~Ma and L.~Zhang.
\newblock On an exact penalty function method for nonlinear mixed discrete
  programming problems and its applications in search engine advertising
  problems.
\newblock {\em Appl. Math. Comput.}, 271:642--656, 2015.

\bibitem{ShapiroSun}
A.~Shapiro and J.~Sun.
\newblock Some properties of the augmented {L}agrangian in cone constrained
  optimization.
\newblock {\em Math. Oper. Res.}, 29:479--491, 2004.

\bibitem{SunSunZhang2008}
D.~Sun, J.~Sun, and L.~Zhang.
\newblock The rate of convergence of the augmented {L}agrangian method for
  nonlinear semidefinite programming.
\newblock {\em Math. Program.}, 114:349--391, 2008.

\bibitem{SunZhangWu2006}
J.~Sun, L.~W. Zhang, and Y.~Wu.
\newblock Properties of the augmented {L}agrangian in nonlinear semidefinite
  optimization.
\newblock {\em J. Optim. Theory Appl.}, 129:437--456, 2006.

\bibitem{WangMaZhou}
C.~Wang, C.~Ma, and J.~Zhou.
\newblock A new class of exact penalty functions and penalty algorithms.
\newblock {\em J. Glob. Optim.}, 58:51--73, 2014.

\bibitem{WuLuoDingChen2013}
H.~Wu, H.~Luo, X.~Ding, and G.~Chen.
\newblock Global convergence of modified augmented {L}agrangian methods for
  nonlinear semidefintie programming.
\newblock {\em Comput. Optim. Appl.}, 56:531--558, 2013.

\bibitem{WuLuoYang2014}
H.~X. Wu, H.~Z. Luo, and J.~F. Yang.
\newblock Nonlinear separation approach for the augmented {L}agrangian in
  nonlinear semidefinite programming.
\newblock {\em J. Glob. Optim.}, 59:695--727, 2014.

\bibitem{ZhaoSunToh2010}
X.~Y. Zhao, D.~Sun, and K.-C. Toh.
\newblock A {N}ewton-{C}{G} augmented {L}agrangian method for semidefinite
  programming.
\newblock {\em SIAM J. Optim.}, 20:1737--1765, 2010.

\bibitem{ZhengZhang2015}
F.~Zheng and L.~Zhang.
\newblock Constrained global optimization using a new exact penalty function.
\newblock In D.~Gao, N.~Ruan, and W.~Xing, editors, {\em Advances in Global
  Optimization}, pages 69--76. Springer, Cham, 2015.

\end{thebibliography}

\end{document}